\documentclass[english,american]{article}
\usepackage[T1]{fontenc}
\usepackage[latin9]{inputenc}
\usepackage[a4paper]{geometry}
\usepackage{xcolor}
\usepackage{babel}
\usepackage{array}
\usepackage{verbatim}
\usepackage{enumitem}
\usepackage{multirow}
\usepackage{amsmath}
\usepackage{amsthm}
\usepackage{amssymb}
\PassOptionsToPackage{normalem}{ulem}
\usepackage{ulem}
\usepackage[unicode=true,pdfusetitle,
 bookmarks=true,bookmarksnumbered=true,bookmarksopen=false,
 breaklinks=false,pdfborder={0 0 1},backref=false,colorlinks=false]
 {hyperref}
\usepackage{breakurl}

\makeatletter

\providecommand{\tabularnewline}{\\}
\providecolor{lyxadded}{rgb}{0,0,1}
\providecolor{lyxdeleted}{rgb}{1,0,0}

\DeclareRobustCommand{\lyxsout}[1]{\ifx\\#1\else\sout{#1}\fi}

\theoremstyle{plain}
\newtheorem{thm}{\protect\theoremname}[section]
  \theoremstyle{definition}
  \newtheorem{defn}[thm]{\protect\definitionname}
 \theoremstyle{plain}
 \newtheorem{observation}[thm]{\protect\observationname}
 \newcommand\thmsname{\protect\theoremname}
 \newcommand\nm@thmtype{theorem}
 \theoremstyle{plain}
 
 \newenvironment{namedthm}[1][Undefined Theorem Name]{
   \ifx{#1}{Undefined Theorem Name}\renewcommand\nm@thmtype{theorem*}
   \else\renewcommand\thmsname{#1}\renewcommand\nm@thmtype{namedtheorem}
   \fi
   \begin{\nm@thmtype}}
   {\end{\nm@thmtype}}
  \theoremstyle{plain}
  \newtheorem{lem}[thm]{\protect\lemmaname}
  \theoremstyle{definition}
  \newtheorem{example}[thm]{\protect\examplename}
  \theoremstyle{plain}
  \newtheorem{prop}[thm]{\protect\propositionname}
 \theoremstyle{definition}
 
 \newenvironment{nameddef}[1][Undefined Theorem Name]{
   \ifx{#1}{Undefined Theorem Name}\renewcommand\nm@thmtype{definition*}
   \else\renewcommand\thmsname{#1}\renewcommand\nm@thmtype{nameddefinition}
   \fi
   \begin{\nm@thmtype}}
   {\end{\nm@thmtype}}
  \theoremstyle{plain}
  \newtheorem{cor}[thm]{\protect\corollaryname}
  \theoremstyle{remark}
  \newtheorem*{rem*}{\protect\remarkname}

\@ifundefined{date}{}{\date{}}
\usepackage{calc}

\usepackage{caption}
\usepackage{tikz}
\pgfrealjobname{diamond_free_families}
\usetikzlibrary{datavisualization}
\usetikzlibrary{datavisualization.formats.functions}

\newif\ifcompiletikz
\newenvironment{tikzwhenenabled}{\ifcompiletikz%
\def\tikzwhenenabledend{}%
\else%
\emph{TikZ figures have been disabled.}
\let\tikzwhenenabledend\endcomment%
\comment%
\fi}{\tikzwhenenabledend}

\providecommand\phantomsection{}

\usepackage{chngcntr}

\newenvironment{boldproof}[1][\proofname] {\par\pushQED{\qed}\normalfont\topsep6\p@\@plus6\p@\relax\trivlist\item[\hskip\labelsep\bfseries#1\@addpunct{.}]\ignorespaces}{\popQED\endtrivlist\@endpefalse}

\@ifundefined{showcaptionsetup}{}{%
 \PassOptionsToPackage{caption=false}{subfig}}
\usepackage{subfig}
\makeatother

  \addto\captionsamerican{\renewcommand{\corollaryname}{Corollary}}
  \addto\captionsamerican{\renewcommand{\definitionname}{Definition}}
  \addto\captionsamerican{\renewcommand{\examplename}{Example}}
  \addto\captionsamerican{\renewcommand{\lemmaname}{Lemma}}
  \addto\captionsamerican{\renewcommand{\propositionname}{Proposition}}
  \addto\captionsamerican{\renewcommand{\remarkname}{Remark}}
  \addto\captionsamerican{\renewcommand{\theoremname}{Theorem}}
  \addto\captionsenglish{\renewcommand{\corollaryname}{Corollary}}
  \addto\captionsenglish{\renewcommand{\definitionname}{Definition}}
  \addto\captionsenglish{\renewcommand{\examplename}{Example}}
  \addto\captionsenglish{\renewcommand{\lemmaname}{Lemma}}
  \addto\captionsenglish{\renewcommand{\propositionname}{Proposition}}
  \addto\captionsenglish{\renewcommand{\remarkname}{Remark}}
  \addto\captionsenglish{\renewcommand{\theoremname}{Theorem}}
  \providecommand{\corollaryname}{Corollary}
  \providecommand{\definitionname}{Definition}
  \providecommand{\examplename}{Example}
  \providecommand{\lemmaname}{Lemma}
  \providecommand{\propositionname}{Proposition}
  \providecommand{\remarkname}{Remark}
  \providecommand{\theoremname}{Theorem}
 \addto\captionsamerican{\renewcommand{\observationname}{Observation}}
 \addto\captionsenglish{\renewcommand{\observationname}{Observation}}
 \providecommand{\observationname}{Observation}
\providecommand{\theoremname}{Theorem}

\begin{document}

\compiletikztrue 

\global\long\def\F{\mathcal{F}}
\global\long\def\X{\mathcal{X}}
\global\long\def\B{\mathcal{\mathcal{\mathcal{B}}}}
\global\long\def\C{\mathcal{C}}
\global\long\def\Y{\mathcal{Y}}
\global\long\def\Z{\mathcal{Z}}
\global\long\def\A{\mathcal{A}}
\global\long\def\D{\mathcal{D}}
\global\long\def\Q{\mathcal{Q}}
\global\long\def\h{\mathcal{H}}
\global\long\def\tif{\textnormal{if }}
\global\long\def\tand{\textnormal{ and }}
\global\long\def\tor{\textnormal{ or }}
\global\long\def\La{\operatorname{La}}
\global\long\def\naturals{\mathbf{N}}
\global\long\def\reals{\mathbf{R}}
\global\long\def\mathnote#1{}

\author{D\'aniel Gr\'osz\thanks{Department of Mathematics, University of Pisa. e-mail: \protect\href{mailto:groszdanielpub@gmail.com}{groszdanielpub@gmail.com}}\and Abhishek
Methuku\thanks{Department of Mathematics, Central European University, Budapest,
Hungary. e-mail: \protect\href{mailto:abhishekmethuku@gmail.com}{abhishekmethuku@gmail.com}}\and Casey Tompkins\thanks{Alfr\'ed R\'enyi Institute of Mathematics, Hungarian Academy of
Sciences. e-mail: \protect\href{mailto:ctompkins496@gmail.com}{ctompkins496@gmail.com}}}

\title{An upper bound on the size of diamond-free families of sets}
\maketitle
\begin{abstract}
Let $\La(n,P)$ be the maximum size of a family of subsets of $[n]=\{1,2,\ldots,n\}$
not containing $P$ as a (weak) subposet. The diamond poset, denoted
$\Q_{2}$, is defined on four elements $x,y,z,w$ with the relations
$x<y,z$ and $y,z<w$. $\La(n,P)$ has been studied for many posets;
one of the major open problems is determining $\La(n,\Q_{2})$. It
is conjectured that $\La(n,\Q_{2})=(2+o(1))\binom{n}{\left\lfloor n/2\right\rfloor }$,
and infinitely many significantly different, asymptotically tight
constructions are known.

Studying the average number of sets from a family of subsets of $[n]$
on a maximal chain in the Boolean lattice $2^{[n]}$ has been a fruitful
method. We use a partitioning of the maximal chains and introduce
an induction method to show that $\La(n,\Q_{2})\leq(2.20711+o(1))\binom{n}{\left\lfloor n/2\right\rfloor }$,
improving on the earlier bound of $(2.25+o(1))\binom{n}{\left\lfloor n/2\right\rfloor }$
by Kramer, Martin and Young.
\end{abstract}

\section{Introduction}

Let $[n]=\{1,2,\dots,n\}$. The Boolean lattice $2^{[n]}$ is defined
as the family of all subsets of $[n]=\{1,2,\ldots,n\}$, and the\emph{
$i$th level} of $2^{[n]}$ refers to the collection of all sets of
size $i$. In 1928, Sperner proved the following well-known theorem.
\begin{thm}[Sperner \cite{sperner1928satz}]
\emph{}If $\F$ is a family of subsets of $[n]$ such that no set
contains another ($A,B\in\F$ implies $A\not\subset B$), then $\left|\F\right|\le\binom{n}{\left\lfloor n/2\right\rfloor }$.
Moreover, equality occurs if and only if $\F$ is a level of maximum
size in $2^{[n]}$. 
\end{thm}

\begin{defn}
Let $P$ be a finite poset, and $\F$ be a family of subsets of $[n]$.
We say that $P$ is contained in $\F$ as a (weak) subposet if there
is an injection $\varphi:P\rightarrow\mathcal{F}$ satisfying $x_{1}<_{p}x_{2}\Rightarrow\varphi(x_{1})\subset\varphi(x_{2})$
for every $x_{1},x_{2}\in P$. $\F$ is called $P$-free if $P$ is
not contained in $\F$ as a weak subposet. We define the corresponding
extremal function as $\La(n,P):=\max\{\left|\F\right|:\mathcal{F}\textnormal{ is \ensuremath{P}-free}\}$. 
\end{defn}

A $k$-chain, denoted by $P_{k}$, is defined to be the poset on the
set $\{x_{1},x_{2},\dots,x_{k}\}$ with the relations $x_{1}\le x_{2}\le\dots\le x_{k}$.
Using the above notation, Sperner's theorem can be stated as $\La(n,P_{2})=\binom{n}{\left\lfloor n/2\right\rfloor }$.
Let $\Sigma(n,k)$ denote the sum of the $k$ largest binomial coefficients
of order $n$. An important generalization of Sperner's theorem due
to Erd\H os\emph{ }\cite{erdos1945lemma} states that $\La(n,P_{k+1})=\Sigma(n,k)$.
Moreover, equality occurs if and only if $\F$ is the union of $k$
of the largest levels in $2^{[n]}$.
\begin{defn}[Posets $\Q_{2},V$ and $\Lambda$]
 The diamond poset, denoted $\Q_{2}$ (or $\D_{2}$ or $\B_{2}$),
is a poset on four elements $\{x,y,z,w\}$, with the relations $x<y,z$
and $y,z<w$. That is, $\Q_{2}$ is a subposet of a family of sets
$\A$ if there are different sets $A,B,C,D\in\A$ with $A\subset B,C$
and $B,C\subset D$. (Note that $B$ and $C$ are not necessarily
unrelated.) The $V$ poset is a poset on $\{x,y,z\}$ with the relations
$x\leq y,z$; the $\Lambda$ poset is defined on $\{x,y,z\}$ with
the relations $x,y\leq z$. That is, the $\Lambda$ is a subposet
of a family of sets $\A$ if there are different sets $B,C,D\in\A$
with $B,C\subset D$.
\end{defn}

The general study of forbidden poset problems was initiated in the
paper of Katona and Tarj\'an \cite{katona1983extremal} in 1983.
They determined the size of the largest family of sets containing
neither a $V$ nor a $\Lambda$. They also gave an estimate on the
maximum size of $V$-free families: $\left(1+\frac{1}{n}+o\!\left(\frac{1}{n}\right)\right)\binom{n}{\left\lfloor n/2\right\rfloor }\le\La(n,V)\le\left(1+\frac{2}{n}\right)\binom{n}{\left\lfloor n/2\right\rfloor }$.
This result was later generalized by De Bonis and Katona \cite{krfork}
who obtained bounds for the $r$-fork poset, $V_{r}$ defined by the
relations $x\le y_{1},y_{2},\dots,y_{r}$. Other posets for which
$\La(n,P)$ has been studied include complete two level posets, batons
\cite{thanh}, crowns $O_{2k}$ (cycle of length $2k$ on two levels,
asymptotically solved except for $k\in\{3,5\}$ \cite{evencrown,lu2014crown}),
butterfly \cite{de2005largest}, skew-butterfly \cite{methuku2014exact},
the $\mbox{N}$ poset \cite{griggs2008}, harp posets $\mathcal{H}(l_{1},l_{2},\dots,l_{k})$,
defined by $k$ chains of length $l_{i}$ between two fixed elements
\cite{griggs2012diamond}, and recently the complete 3 level poset
$K_{r,s,t}$ \cite{patkos} among others. (See \cite{griggs_survey}
for a nice survey by Griggs and Li.) 

One of the first general results is due to Bukh \cite{bukh2009set}
who determined the asymptotic value of $\La(n,P)$ for all posets
whose Hasse diagram is a tree: If $T$ is a finite poset whose Hasse
diagram is a tree of height $h(T)\ge2$, then $\La(n,T)=(h(T)-1)\binom{n}{\lfloor n/2\rfloor}\left(1+O\!\left(\frac{1}{n}\right)\right).$

Using more general structures instead of chains for double counting,
Burcsi and Nagy \cite{burcsi2013method} obtained a weaker version
of this theorem for general posets showing that $\La(n,P)\le\left(\frac{\left|P\right|+h(P)}{2}-1\right)\binom{n}{\lfloor n/2\rfloor}$.
Later this was generalized by Chen and Li \cite{chen2014note} and
recently this general bound was improved by the authors of the present
article \cite{GrMeToGeneral}. 

The most investigated poset for which even the asymptotic value of
$\La(n,P)$ has yet to be determined is the diamond $\Q_{2}$ which
is the topic of our paper. The two middle levels of the Boolean lattice
do not contain a diamond, so $\La(n,\Q_{2})\geq(2-o(1))\binom{n}{\left\lfloor n/2\right\rfloor }$.
Czabarka, Dutle, Johnston and Sz\'ekely \cite{diamond_constructions}
gave infinitely many asymptotically tight constructions by using random
set families defined from posets based on Abelian groups. Such constructions
suggest that the diamond problem is hard. Using a simple and elegant
argument, Griggs, Li and Lu \cite{griggs2012diamond} showed that
$\La(n,\Q_{2})<2.296\mbox{\ensuremath{\binom{n}{\left\lfloor n/2\right\rfloor }}}$.
Some time after they had announced this bound, Axenovich, Manske and
Martin \cite{diamond_228} improved the upper bound to $2.283\mbox{\ensuremath{\binom{n}{\left\lfloor n/2\right\rfloor }}}$.
This bound was further improved to $2.273\binom{n}{\left\lfloor n/2\right\rfloor }$
by Griggs, Li and Lu \cite{griggs2012diamond}. The best known upper
bound on $\La(n,\Q_{2})$ is $(2.25+o(1))\binom{n}{\left\lfloor n/2\right\rfloor }$
due to Kramer, Martin and Young \cite{diamond225}.
\begin{defn}
A \emph{maximal chain} or, for the rest of this article, simply a
\emph{chain} of the Boolean lattice is a sequence of sets $\emptyset,\{x_{1}\},\{x_{1},x_{2}\},\{x_{1},x_{2},x_{3}\},\dots,[n]$
with $x_{1},x_{2},x_{3}\ldots\in[n]$. We refer to $\{x_{1},\ldots,x_{i}\}$
as the $i$th set on the chain. In particular, we refer to $\{x_{1}\}$
as the first set on the chain, or just say that the chain starts with
the element $x_{1}$ (as a singleton). We refer to $x_{i}$ as the
$i$th element added to form the chain.
\end{defn}

\begin{defn}
The \emph{Lubell function} of a family of sets $\F\subseteq2^{[n]}$
is defined as 
\[
l(n,\F)=\sum_{F\in\F}\frac{1}{\binom{n}{\left|F\right|}}.
\]
 The notation is shortened to just $l(\F)$ when there is no ambiguity
as to the dimension of the Boolean lattice.
\end{defn}

\begin{observation}
The Lubell function of a family $\F$ is the average number of sets
from $\F$ on a chain, taken over all $n!$ chains. In particular,
the Lubell function of a level is 1, and the Lubell function of an
antichain $\F$ is the number of chains containing a set from $\F$
divided by $n!$. The Lubell function is additive across a union of
disjoint families of sets. Furthermore, $\left|\F\right|\leq l(\F)\binom{n}{\left\lfloor n/2\right\rfloor }$
(\cite{LYMlubell}).
\end{observation}

The Lubell function was derived from the celebrated YMBL inequality
which was independently discovered by Yamamoto, Meshalkin, Bollob\'as
and Lubell. Using the Lubell function terminology, it states that
\begin{namedthm}[YMBL inequality \textup{\textmd{(Yamamoto, Meshalkin, Bollob\'as,
Lubell \cite{yamamoto1954logarithmic,meshalkin1963generalization,bollobas1965generalized,LYMlubell})}}]
If $\F\subseteq2^{[n]}$ is an antichain, then $l(\F)\leq1$.
\end{namedthm}
For a poset $P$, let $\overline{l}(n,P)$ be the maximum of $l(n,\F)$
over all families $\F\subseteq2^{[n]}$ which are both $P$-free and
contain the empty set. Let $\overline{l}(P)=\limsup_{n\to\infty}\overline{l}(n,P)$.
Griggs, Li and Lu proved that 
\begin{lem}[Griggs, Li and Lu \cite{griggs2012diamond}]
\label{lem:minpartition}
\[
\La(n,\Q_{2})\leq\left(\overline{l}(\Q_{2})+o(1)\right)\binom{n}{\left\lfloor n/2\right\rfloor }.
\]
\end{lem}

Kramer, Martin and Young used flag algebras to prove that
\begin{lem}[Kramer, Martin and Young \cite{diamond225}]
\label{lem:lubell2.25}$\overline{l}(\Q_{2})=2.25$
\end{lem}

thereby proving
\begin{thm}[Kramer, Martin and Young \cite{diamond225}]
\label{thm:La2.25}
\[
\La(n,\Q_{2})\leq(2.25+o(1))\binom{n}{\left\lfloor n/2\right\rfloor }.
\]
\end{thm}

The following construction shows that $\bar{l}(\Q_{2})\ge2.25$ in
Lemma \ref{lem:lubell2.25}. There are other constructions known as
well.
\begin{example}
Let $\F\subseteq2^{[n]}$ consist of all the sets of the following
forms: $\emptyset,\{e\},\{e,o\},\{o_{1},o_{2}\}$ where $e$ denotes
any even number in $[n]$, and $o$, $o_{1}$ and $o_{2}$ denote
any odd numbers in $[n]$. This family is diamond-free, and $l(\F)=2.25\pm o(1)$.
\end{example}

\begin{example}
\label{exa:canonical}This construction is a generalization of the
previous one. Let $A\subseteq[n]$ with $\left|A\right|=an$. Let
$\F\subseteq2^{[n]}$ consist of all the sets of the following forms:
$\emptyset,\{e\},\{e,o\},\{o_{1},o_{2}\}$ where now $e$ denotes
any element of $A$, while $o$, $o_{1}$ and $o_{2}$ denote any
elements of $[n]\setminus A$. This family is diamond-free, and $l(\F)=2+a-a^{2}\pm o(1)$.
This family contains all size 2 sets that do not form a diamond with
$\emptyset$ and the singletons, so all maximal diamond-free families
on levels 0, 1 and 2 that contain $\emptyset$ are of this form.
\end{example}

The following restriction of the problem of diamond-free families
has been investigated: How big can a diamond-free family be if it
can only contain sets from the middle three levels of $2^{[n]}$ (denoted
$\B(n,3)$)? Better bounds are known with this restriction. Axenovich,
Manske and Martin showed that
\begin{thm}[Axenovich, Manske and Martin \cite{diamond_228}]
\label{thm:3levels2.207}If $\F\subseteq\B(n,3)$ is diamond-free,
then $\left|\F\right|\leq\bigl(2.20711+o(1)\bigr)\binom{n}{\left\lfloor n/2\right\rfloor }$.
\end{thm}

Later, Manske and Shen improved it to $2.1547\mbox{\ensuremath{\binom{n}{\left\lfloor n/2\right\rfloor }}}$
in \cite{manske2013three} and recently, Balogh, Hu, Lidick\'y and
Liu gave the best known bound of $2.15121\binom{n}{\left\lfloor n/2\right\rfloor }$
in \cite{diamond21512} using flag algebras.
\begin{defn}
We call a chain \emph{maximal\textendash non-maximal (MNM)} with respect
to (w.r.t.) $\F$ if it contains a set from $\F$, and the biggest
set contained in $\F$ on the chain is not maximal in $\F$ (i.e.,
there are other sets from $\F$ containing it on some other chains).
\end{defn}

It is easy to see that an $\emptyset$-free family is $\Lambda$-free
if and only if the family we get by adding $\emptyset$ is diamond-free;
adding $\emptyset$ increases the Lubell function by 1. In Section
\ref{sec:invertedV} of this paper, we prove the following lemma:
\begin{lem}
\label{lem:lubell}Let $\F\subseteq2^{[n]}$ be a $\Lambda$-free
family that does \emph{not} contain the empty set, nor any set of
size bigger than $n-n'$ for some $n'\in\naturals$ (that can be chosen
independently of $n$). Assume that there are $cn!$ MNM chains w.r.t.\ $\F$.
Then $l(\F)\leq1-\min\!\left(c+\frac{1}{n'},\frac{1}{4}\right)+\sqrt{\min\!\left(c+\frac{1}{n'},\frac{1}{4}\right)}+\frac{3}{n'}$.
\end{lem}

It is easy to see that in Example \ref{exa:canonical} the number
of MNM-chains is approximately $a^{2}n!$ (so $a\approx\sqrt{c}$):
these are the chains whose second set is $\{e_{1},e_{2}\}$ with $e_{1},e_{2}\in A$.
Thus, this lemma is (asymptotically) sharp, and states that for a
given number of MNM chains, Example \ref{exa:canonical} cannot be
beaten (with some restriction on the sizes of the sets). Barring the
requirement that the topmost $n'$ levels be empty, Lemma \ref{lem:lubell}
is a generalization of Lemma \ref{lem:lubell2.25}. The proof of Lemma
\ref{lem:minpartition} in \cite{diamond225} actually works with
the restriction of Lemma \ref{lem:lubell} concerning the topmost
sets (that there is no set of size bigger than $n-n'$) with $n'=n/2-n^{2/3}$,
immediately giving a new proof of Theorem \ref{thm:La2.25}. Our proof
of Lemma \ref{lem:lubell} includes an intricate induction step and
a (non-combinatorial) lemma about functions involving a lot of elementary
algebra and calculus; but it does not require flag algebras, and it
does not use details of the structure of $\F$ above the second level
(except inside the induction).

Section \ref{sec:diamond} of this paper uses Lemma \ref{lem:lubell}
to prove our main theorem:
\begin{thm}
\label{thm:main}$\La(n,\Q_{2})\leq\left(\frac{\sqrt{2}+3}{2}+o(1)\right)\binom{n}{\left\lfloor n/2\right\rfloor }<\bigl(2.20711+o(1)\bigr)\binom{n}{\left\lfloor n/2\right\rfloor }$.
\end{thm}

The proof is inspired by the proof of Theorem \ref{thm:3levels2.207}
(the same bound when restricted to 3 levels) as in \cite{diamond_228},
using the idea of grouping chains by the smallest set contained in
$\F$ on a chain (as developed in \cite{griggs2012diamond} and \cite{diamond225}).\global\let\addcontentslinesave\addcontentsline\renewcommand{\addcontentsline}[3]{}

\section{$\Lambda$-free families \textendash{} Proof of Lemma \ref{lem:lubell}\label{sec:invertedV}\addcontentslinesave{toc}{section}{\ref{sec:invertedV}
Inverted V-free families \textendash{} Proof of Lemma \ref{lem:lubell}}}

\subsection{Definitions and main lemma}
\begin{defn}
\global\let\addcontentsline\addcontentslinesave\label{def:functions}We
define the following functions:%

\begin{itemize}
\item For $x\in[0,1],c\in[0,\infty)$,
\[
f(x,c)=\begin{cases}
1-x+\left(\frac{1}{4\left(x-x^{2}\right)}-1\right)c & \tif x\leq\frac{1}{2}\tand c<4\left(x-x^{2}\right)^{2}\\
x^{2}-2x+1-c+\sqrt{c} & \tif x\leq\frac{1}{2}\tand4\left(x-x^{2}\right)^{2}\leq c\leq\frac{1}{4}\\
x^{2}-2x+1.25 & \tif x\leq\frac{1}{2}\tand\frac{1}{4}\leq c\\
1-x & \tif\frac{1}{2}\leq x.
\end{cases}
\]
\begin{tikzwhenenabled}
\begin{figure}
\subfloat[Values of $f(x,c)$ plotted in $x$, for $c=0,0.0125,0.025,\ldots,0.25$
(bottom to top). Note that the $x\geq0.5$ part of the plots coincide.]{\beginpgfgraphicnamed{diamond_free_families_figure1}
\tikz
\datavisualization
	[scientific axes, x axis={length=6cm,ticks and grid={step=0.25}}, y axis={length=8.5cm,ticks and grid={step=0.25}}, visualize as smooth line/.list={100,201.0,202.0,....0,220.0,301.0,302.0,....0,319.0}]
data [set=100,format=function] {
	var c : {0};
	var x_ : interval [0:0.6] samples 2;
	func x = \value{x_};
	func y = 1-\value{x_};
}
data [format=function] {
	var i : {1,2,...,20};
	func c = \value i/80;
	var x_ : interval [0:(1-sqrt(1-2*sqrt(\value c)))/2] samples 11;
	func x = \value{x_};
	func y = \value{x_}^2-2*\value{x_}+1-\value c+sqrt(\value c);
	func set = 200+\value i;
}
data [format=function] {
	var i : {1,2,...,19};
	func c = \value i/80;
	var x_ : interval [(1-sqrt(1-2*sqrt(\value c)))/2:0.5] samples 11;
	func x = \value{x_};
	func y = 1-\value{x_}+(1/(4*(\value{x_}-\value{x_}^2))-1)*\value c;
	func set = 300+\value i;
};
\endpgfgraphicnamed}\hspace*{\fill}\subfloat[Values of $f(x,c)$ plotted in $c$, for $x=0,0.05,0.01,\ldots,1$
(top to bottom).]{\beginpgfgraphicnamed{diamond_free_families_figure2}
\tikz
\datavisualization
	[scientific axes, x axis={length=8cm,ticks and grid={step=0.25}}, y axis={length=10cm,ticks and grid={step=0.25}}, visualize as smooth line/.list={110.0,111.0,....0,120.0,200.0,201.0,....0,209.0,300.0,301.0,....0,309.0,401.0,402.0,....0,409.0}]
data [format=function] {
	var i : {10,11,...,20};
	func x_ = \value i/20;
	var c : {0,1)};
	func x = \value{c};
	func y = 1-\value{x_};
	func set = 100+\value i;
}
data [format=function] {
	var i : {1,2,...,9};
	func x_ = \value i/20;
	var c : {0,4*(\value{x_}-\value{x_}^2)^2};
	func x = \value{c};
	func y = 1-\value{x_}+(1/(4*(\value{x_}-\value{x_}^2))-1)*\value c;
	func set = 400+\value i;
}
data [format=function] {
	var i : {0,1,...,9};
	func x_ = \value i/20;
	var c : interval[4*(\value{x_}-\value{x_}^2)^2:0.25] samples 20-2*\value i;
	func x = \value{c};
	func y = \value{x_}^2-2*\value{x_}+1-\value c+sqrt(\value c);
	func set = 200+\value i;
}
data [format=function] {
	var i : {0,1,...,9};
	func x_ = \value i/20;
	var c : interval[0.25:1] samples 2;
	func x = \value{c};
	func y = \value{x_}^2-2*\value{x_}+1.25;
	func set = 300+\value i;
};
\endpgfgraphicnamed}
\end{figure}
\end{tikzwhenenabled}
\item For $x\in[0,1),c\in[0,\infty),a\in[0,1-x),\tilde{a}\in\left[0,\min\!\left(a,\frac{c}{x+a}\right)\right]$,
\[
g(x,c,a,\tilde{a})=a+(1-x-a)f\!\left(x+a,\frac{c-\tilde{a}(x+a)}{1-x-a}\right)+2\tilde{a}(1-x-a).
\]
\item For $x\in[0,1),c\in[0,\infty),a\in[0,1-x),\tilde{a}\in\left[0,\min\!\left(a,\frac{c}{x+a}-x\right)\right]$,
\[
h(x,c,a,\tilde{a})=a+(1-x-a)f\!\left(x+a,\frac{c-(x+\tilde{a})(x+a)}{1-x-a}\right)+2\tilde{a}(1-x-a)+x-3x(x+a).
\]
\end{itemize}
\end{defn}

\begin{lem}
\label{lem:functions}The functions above satisfy the following conditions:

\begin{enumerate}
\item \label{enu:x=00003D0}For all $c\in[0,\infty)$, if $\tilde{c}=\min\!\left(c,\frac{1}{4}\right)$,
then $f(0,c)=1-\tilde{c}+\sqrt{\tilde{c}}$.
\item \label{enu:concave}$f(x,c)$ is concave and monotonously increasing
in $c$, and monotonously decreasing in $x$.
\item \label{enu:g}For all $x\in[0,1],c\in[0,\infty),a\in[0,1-x),\tilde{a}\in\left[0,\min\!\left(a,\frac{c}{x+a}\right)\right]:g(x,c,a,\tilde{a})\leq f(x,c)$.
\item \label{enu:h}For all $x\in[0,1],c\in[0,\infty),a\in[0,1-x),\tilde{a}\in\left[0,\min\!\left(a,\frac{c}{x+a}-x\right)\right]:h(x,c,a,\tilde{a})\leq f(x,c)$.
\item \label{enu:1-x}For all $c\in[0,\infty):1-x\leq f(x,c)$.
\end{enumerate}
\end{lem}

We prove Lemma \ref{lem:functions} in Appendix \ref{sec:functions_proof}.

Rather that proving Lemma \ref{lem:lubell} directly, we prove a strengthening
of it \textendash{} Lemma \ref{lem:ind}. This strengthened version
involves additional parameters, $X$ and $\X$, and their functions
$x$, $\alpha$ and $\mu$, which we introduce in order to make the
inductive proof possible. Lemma \ref{lem:lubell} is a special case
of Lemma \ref{lem:ind} with $X=\X=\emptyset$. In the rest of Section
\ref{sec:invertedV}, we prove Lemma \ref{lem:ind}.

\begin{lem}
\label{lem:ind}Let $\F\subseteq2^{[n]}$ be a $\Lambda$-free family
which does not contain $\emptyset$, nor any set larger than $n-n'$
for some $n'\in\naturals$. Let us assume that we are given a ``forbidden''
set $X\subseteq[n]$, with $x=\frac{\left|X\right|}{n}$. Also, let
$\X\subset2^{[n]}$ be a ``forbidden'' antichain in which each set
contains exactly one element of $X$ (and may or may not be a singleton).
Let us assume that the sets in $\F$ are disjoint from $X$, and unrelated
to every set in $\X$. Let $\alpha=l(\X)$, and let $\mu n!$ be the
number of chains which start with an element of $X$ as a singleton,
but do not contain any set in $\X$. Assume, furthermore, that there
are $cn!$ MNM chains w.r.t.\ $\F$. Then $l(\F)\leq f\!\left(x,c+\mu+\frac{1}{n'}\right)-(\alpha-\mu-x)+\frac{3}{n'}$.
\end{lem}

First we verify the base case of the induction.
\begin{prop}
\label{prop:indbase}Lemma \ref{lem:ind} holds for $n\leq n'$.
\end{prop}

\begin{proof}
$\F=\emptyset$. $\X$ is an antichain, so, by the YMBL inequality,
$\alpha\leq1$. By Lemma \ref{lem:functions} Point \ref{enu:concave}\ and
Point \ref{enu:1-x}, $f\!\left(x,c+\mu+\frac{1}{n'}\right)-(\alpha-\mu-x)+\frac{3}{n'}\geq f(x,c+\mu)-(\alpha-\mu-x)\geq1-x-(\alpha-x)\geq0=l(\F)$.
\end{proof}
From now on we assume $n'\leq n-1$.
\begin{nameddef}[Notation]
Let $A\subseteq[n]\setminus X$ be the set of elements of $[n]$
that appear as singletons in $\F$, and let $a=\frac{\left|A\right|}{n}$.
Let $\B$ be the family of those sets in $\F$ which contain at least
one element of $A$, but which are not singletons. Let $\beta$ be
the Lubell function of $\B$. Let $\mathcal{C}$ be the family of
those sets in $\F$ which only contain elements of $[n]\setminus X\setminus A$.

Let $\tilde{A}=\bigl\{ e\in A:(\exists B\in\B:e\in B)\bigr\}$, and
let $\tilde{a}=\frac{\left|\tilde{A}\right|}{n}$. Let $c_{0}n!$
be the number of chains that start with $\{e\}$ as a singleton for
some $e\in\tilde{A}$, but do not contain any set from $\B$. Let
$\nu n!$ be the number of chains that start with $\{e\}$ as a singleton
for some $e\in\tilde{A}$, continue with an element of $[n]\setminus X\setminus A$
as the second element added to form the chain, yet do not contain
any set from $\B$.

\global\long\def\One{\overline{1}}
\global\long\def\one{\underline{1}}
Let $\One=\frac{n}{n-1}>1$ and $\one=\frac{(x+a)n-1}{(x+a)(n-1)}\leq1$.
These correction factors will account for the difference from the
asymptotic behavior. (They are both typically close to 1. If $x+a=0$,
let $\one=1$; it is irrelevant as it will always be multiplied by
$x+a$.)%
\end{nameddef}

\begin{table}
\caption{Summary of notation}
\centering{}%
\begin{tabular}{cll|l}
$X$ & \multicolumn{2}{>{\raggedright}p{0.7\columnwidth}|}{``Forbidden'' set (sets in $\F$ are disjoint from it \textendash{}
parameter of Lemma \ref{lem:ind})} & $x=\left|X\right|/n$\tabularnewline
$\X$ & \multicolumn{2}{>{\raggedright}p{0.7\columnwidth}|}{``Forbidden'' antichain (sets in $\F$ are unrelated to sets in
it \textendash{} parameter of Lemma \ref{lem:ind})} & $\alpha=l(\X)$\tabularnewline
$\mu$ & \multicolumn{2}{>{\raggedright}p{0.7\columnwidth}}{$\#\left\{ \text{chains containing \ensuremath{\{d\}} for \ensuremath{d\in X} but no set from \ensuremath{\X}}\right\} /n!$} & \tabularnewline
$c$ & \multicolumn{2}{>{\raggedright}p{0.7\columnwidth}}{$\#\{\text{MNM chains}\}/n!$ (parameter of Lemma \ref{lem:lubell}
/ Lemma \ref{lem:ind})} & \tabularnewline
\hline 
$A$ & $\text{\ensuremath{\left\{ e\in[n]:\{e\}\in\F\right\} }}$ & \multirow{3}{*}{$\left.\vphantom{\mbox{\begin{tabular}{c}
 \ensuremath{\left\{  \right\} } \\
 \ensuremath{\left\{  \right\} } \\
 \ensuremath{\left\{  \right\} } 
\end{tabular}}}\right\} \left\{ \{e\}:e\in A\right\} \cup\B\cup\C=\F$} & $a=\left|A\right|/n$\tabularnewline
$\B$ & $\left\{ B\in\F:(\left|B\right|\geq2,A\cap B\neq\emptyset)\right\} $ &  & $\beta=l(\B)$\tabularnewline
$\C$ & $\left\{ C\in\F:C\subseteq[n]\setminus X\setminus A\right\} $ &  & \tabularnewline
$\tilde{A}$ & \multicolumn{2}{>{\raggedright}p{0.7\columnwidth}|}{$\bigl\{ e\in A:(\exists B\in\B:e\in B)\bigr\}$} & $\tilde{a}=\bigl|\tilde{A}\bigr|/n$\tabularnewline
$\nu$ & \multicolumn{2}{>{\raggedright}p{0.7\columnwidth}}{$\#\bigl\{\text{chains containing }\ensuremath{\{e\}}\text{ and }\ensuremath{\{e,o\}}\text{ for }\ensuremath{e\in\tilde{A},o\in[n]\setminus X\setminus A}\text{ but no set from }\B\bigr\}/n!$} & \tabularnewline
$c_{0}$ & \multicolumn{2}{>{\raggedright}p{0.7\columnwidth}}{$\#\bigl\{\text{chains containing \ensuremath{\{e\}} for \ensuremath{e\in\tilde{A}} but no set from \ensuremath{\B}}\bigr\}/n!$} & \tabularnewline
\end{tabular}
\end{table}

\emph{Outline of the proof:} In Subsection \ref{subsec:On-the-structure},
we make some observations on the structure of $\X$ and $\B$. In
Subsection \ref{subsec:Induction}, we will finish the proof by applying
induction to the Boolean lattices $\bigl[\{o_{i}\},[n]\bigr]$ where
$o_{i}\in[n]\setminus X\setminus A$. When applying Lemma \ref{lem:ind}
by induction, we will use $X\cup A$ in the place of $X$, while sets
from $\X$ and $\B$ will contribute to the family we use in the place
of $\X$ (which we will denote by $\X'_{i}$). We know little about
the parameters of each $\X'_{i}$, but we will be able to bound their
sums. The relevant calculations are done in Subsection \ref{subsec:Chain-calculations}.

\subsection{\label{subsec:On-the-structure}On the structure of $\protect\X$
and $\protect\B$}
\begin{prop}
\label{prop:Xform}Every $D\in\X$ is of the form $\{d,o_{1},\ldots,o_{k}\}$
with $d\in X,o_{1},\ldots,o_{k}\in[n]\setminus X\setminus A$ (where
$k$ may be 0).
\end{prop}

\begin{proof}
$D$ contains exactly one element of $X$ by definition. Let $e\in A$;
then $e\notin D$ for otherwise $D$ and $\{e\}\in\F$ would be related.
\end{proof}
\begin{prop}
\label{prop:Bform}Sets in $\B$ only contain one element of $A$.
$\B$ is an antichain, and the sets in $\B$ are also unrelated to
every set in $\C$.
\end{prop}

\begin{proof}
If $e_{1}\in B\in\B$ with $e_{1}\in A$, and B was related to another
set $S\in\F$, then $\{e_{1}\}$, $B$ and $S$ would form a $\Lambda$.
This applies to any $S\in\B\cup\C$, as well as $S=\{e_{2}\}$ for
any $e_{1}\neq e_{2}\in A$.
\end{proof}
\begin{prop}
\label{prop:cbound}$\tilde{a}(x+a)\one\leq\tilde{a}(x+a)\one+\nu=c_{0}\leq c$,
and thus $\tilde{a}\leq\frac{c}{(x+a)\one}$.
\end{prop}

\begin{proof}
Any chain on which the singleton is $\{e\}$ and the second set is
$\{e,d\}$ with $e\in\tilde{A}$ and $d\in X\cup A$ is always an
MNM chain: $\{e,d\}$ and any set that contains it is forbidden from
being in $\B$ either because it is not disjoint from $X$ (when $d\in X$),
or because it would contain two elements of $A$ (when $d\in A$).
The number of such chains is $\tilde{a}n\cdot(an+xn-1)\cdot(n-2)!=\tilde{a}(x+a)n!\one$.
And out of the chains which start with $\{e\}$, and whose second
set is $\{e,o\}$ with some $o\in[n]\setminus X\setminus A$, $\nu n!$
do not contain any set from $\B$.

We have $c_{0}\leq c$ because a chain whose first set is $\{e\}$
for some $e\in\tilde{A}$, but does not contain any set from $\B$,
is an MNM chain.
\end{proof}
For a family of sets $\A\subseteq2^{[n]}$, let $m(\A)n!$ be the
number of chains which start with an element of $X$ as a singleton
and do not contain any set from $\A$. (For example, $m(\X)=\mu$,
and therefore $l(\X)-m(\X)=\alpha-\mu.)$ For a fixed $d\in X$, let
$m_{d}(\A)n!$ be the number of chains on which the singleton is $\{d\}$,
and do not contain any element of $\A$.
\begin{prop}
\label{prop:uniform}For any $d\in X$, let $\X_{d}=\{D\in\X:d\in D\}$.
We can assume without loss of generality that for any $d_{1},d_{2}\in X$,
$\left\{ D\setminus\left\{ d_{1}\right\} :D\in\X_{d_{1}}\right\} =\left\{ D\setminus\left\{ d_{2}\right\} :D\in\X_{d_{2}}\right\} $.
That is, if $\X$ does not satisfy this condition, we show a family
$\hat{\X}$ which does, and also satisfies the conditions of Lemma
\ref{lem:ind}'s statement (each set contains exactly one element
of $X$, the sets are unrelated to each other and to every set in
$\F$), and for which $f\!\left(x,c+m(\hat{\X})+\frac{1}{n'}\right)-\bigl(l(\hat{\X})-m(\hat{\X})-x\bigr)\leq f\!\left(x,c+\mu+\frac{1}{n'}\right)-(\alpha-\mu-x)$.
\end{prop}

\begin{proof}
Let $d_{0}\in X$ be such that 
\begin{gather*}
f\!\left(x,c+\left|X\right|m_{d_{0}}\!\left(\X_{d_{0}}\right)+\frac{1}{n'}\right)-\bigl(\left|X\right|l\!\left(\X_{d_{0}}\right)-\left|X\right|m_{d_{0}}\!\left(\X_{d_{0}}\right)-x\bigr)\\
=\min_{d\in X}\!\left[f\!\left(x,c+\left|X\right|m_{d}\!\left(\X_{d}\right)+\frac{1}{n'}\right)-\bigl(\left|X\right|l\!\left(\X_{d}\right)-\left|X\right|m_{d}\!\left(\X_{d}\right)-x\bigr)\right].
\end{gather*}
 Let $\hat{\X}=\bigl\{ D\setminus\{d_{0}\}\cup\{d\}:d\in X,D\in\X_{d_{0}}\bigr\}$.

$\X=\bigsqcup_{d\in X}\X_{d}$, so $\alpha=\sum_{d\in X}l(\X_{d})$.
It immediately follows from the definition of $\X_{d}$ that if a
chain has $\{d\}$ as a singleton, and does not contain any set from
$\X_{d}$, then it does not contain any set from $\X$. So $\mu=\sum_{d\in X}m_{d}(\X_{d})$.
Similarly, $l(\hat{\X})=\left|X\right|l\!\left(\X_{d_{0}}\right)$
and $m(\hat{\X})=\left|X\right|m_{d_{0}}\!\left(\X_{d_{0}}\right)$.
Since $f(x,c)$ is monotonously increasing and concave in $c$, using
Jensen's inequality 
\begin{gather*}
f\!\left(x,c+\left|X\right|m_{d_{0}}\!\left(\X_{d_{0}}\right)+\frac{1}{n'}\right)-\bigl(\left|X\right|l\!\left(\X_{d_{0}}\right)-\left|X\right|m_{d_{0}}\!\left(\X_{d_{0}}\right)-x\bigr)\\
\leq\frac{1}{\left|X\right|}\left[\sum_{d\in X}f\!\left(x,c+\left|X\right|m_{d}\!\left(\X_{d}\right)+\frac{1}{n'}\right)-\bigl(\left|X\right|\sum_{d\in X}l\!\left(\X_{d}\right)-\left|X\right|\sum_{d\in X}m_{d}\!\left(\X_{d}\right)-\left|X\right|x\bigr)\right]\\
\leq f\!\left(x,c+\mu+\frac{1}{n'}\right)-(\alpha-\mu-x).
\end{gather*}

Sets in $\hat{\X}$ contain exactly one element of $X$, and form
an antichain. They are also unrelated to every set $S\in\F$: $S$
cannot contain any element of $X$, so it could only be related to
a set in $\hat{\X}$ by being its subset. But $S$ must also be unrelated
to every $D\in\X_{d_{0}}\subseteq\X$, so it cannot be a subset of
$D\setminus\{d_{0}\}\cup\{d\}$ either.
\end{proof}
In fact we will only use the following simple corollary of Proposition
\ref{prop:uniform}. In many parts of the rest of this section we
will treat the two cases of the corollary below separately.
\begin{cor}
With the assumption of Proposition \ref{prop:uniform},

\begin{itemize}
\item either $\X=\bigl\{\{d\}:d\in X\bigr\}$ (we refer to it as the \textbf{singletons
case}),
\item or $\X$ does not contain any singleton (referred to as the \textbf{no
singleton case}).
\end{itemize}
\end{cor}

\begin{proof}
Let $d_{1}\in X$. (If $X=\emptyset$, both statements trivially hold.)
If $\{d_{1}\}\in\X_{d_{1}}=\{D\in\X:d_{1}\in D\}$, then $\X_{d_{1}}=\left\{ \{d_{1}\}\right\} $,
because sets in $\X_{d_{1}}$ are unrelated. So either $\X_{d_{1}}=\left\{ \{d_{1}\}\right\} $
or $\X_{d_{1}}$ does not contain any singleton, yielding the two
cases above by Proposition \ref{prop:uniform}.
\end{proof}
\begin{rem*}
The fact that sets in $\X$ contain an element of $X$ implies that
sets in $\F$ do not contain sets in $\X$. Now, let us consider what
restrictions are imposed on $\F$ by the fact that sets in $\F$ are
not contained in the sets in $\X$, beyond the other conditions of
Lemma \ref{lem:ind} (namely that all the sets in $\F$ are disjoint
from $X$).

In the singletons case, clearly there are no such additional restrictions.
However, in the no singleton case, there are two additional restrictions
that are not already implied by the set $X$:

\begin{itemize}
\item The union of singletons in $\F$, $A\subseteq[n]\setminus\bigcup\X$.
\item Sets in $\C$ must not be contained in sets in $\X$. Clearly this
imposes a restriction only if $\X$ contains sets bigger than 2.
\end{itemize}
\end{rem*}
\begin{example}
\label{exa:nosingleton}Let $C\subseteq[n]\setminus X$, and let $\X=\bigl\{\{d,o\}:d\in X,o\in C\bigr\}$.
Then $\alpha=l(\X)=\frac{2\cdot xn\cdot\left|C\right|\cdot(n-2)!}{n!}=2x\frac{\left|C\right|}{n}\One$,
and $\mu=\frac{xn\cdot(xn+an-1)\cdot(n-2)!}{n!}=x(x+a)\one$. The
only restriction on $\F$ that this $\X$ creates is that the union
of singletons $A\subseteq[n]\setminus X\setminus C$.

In other words, let us assume that $\alpha=l(\X)=2x\gamma\One$ for
$\gamma\in\reals$ (without assuming that $\X$ is of the above form).
Then it is possible that $a=\frac{\left|A\right|}{n}$ can be as big
as $1-x-\gamma$ with $\X$ not creating any restrictions on $\C$
(depending on the actual structure of $\X$, namely, if it is made
up of sets of size $2$ as above; then $\alpha=2x(1-x-a)\One$ and
$\mu=x(x+a)\one$). But if $a>1-x-\gamma$, then $\alpha=2x\gamma\One$
implies that $\X$ contains sets bigger than 2, and thus it creates
restrictions on $\C$. So, in the no singleton case, one way to understand
the calculations that follow is to check them for $\X=\bigl\{\{d,o\}:d\in X,o\in C\bigr\}$;
then check what happens if $x$, $c$ and $a$ are fixed, but $\X$
is changed.
\end{example}

\subsection{\label{subsec:Chain-calculations}Chain calculations}

Now we estimate the numbers of certain types of chains, in preparation
for applying induction.
\begin{prop}
\label{prop:Xchainbound}In the no singleton case, $\left(\alpha-x+\mu\right)n!$
chains start with $\{o\}$ for some $o\in[n]\setminus X\setminus A$,
and contain a set from $\X$.
\end{prop}

\begin{proof}
A total of $\alpha n!$ chains contain a set $D\in\X$. By Proposition
\ref{prop:Xform}, the singleton on such a chain is either from $X$
or $[n]\setminus X\setminus A$. The number of chains which start
with an element of $X$ as their singleton and do not contain a set
from $\X$ is $\mu n!$, so the number of chains which contain a set
from $\X$, and which start with an element of $X$, is $(x-\mu)n!$.
On the rest, the singleton is from $[n]\setminus X\setminus A$.
\end{proof}
\begin{prop}
\label{prop:Bchainbound}$\left(\beta-\tilde{a}(1-x-a)\One+\nu\right)n!$
chains start with $\{o\}$ for some $o\in[n]\setminus X\setminus A$,
and contain a set from $\B$.
\end{prop}

\begin{proof}
A total of $\beta n!$ chains contain a set from $\B$. A set in $\B$
is of the form $\{e,o_{1},\ldots,o_{k}\}$ with $e\in\tilde{A},o_{1},\ldots,o_{k}\in[n]\setminus X\setminus A,k\geq1$.
A chain that contains a $B\in\B$, and does not start with $\{o\}$
for some $o\in[n]\setminus X\setminus A$, must start with an element
of $\tilde{A}$, and continue with an element of $[n]\setminus X\setminus A$
as the second element added to form the chain. There are $\tilde{a}n\cdot(1-x-a)n\cdot(n-2)!=\tilde{a}(1-x-a)\One n!$
such chains, out of which $\nu n!$ do not contain any set from $\B$.
So $\left(\tilde{a}(1-x-a)\One-\nu\right)n!$ chains contain a set
from $\B$ and start with an element of $\tilde{A}$. The rest start
with $\{o\}$ for some $o\in[n]\setminus X\setminus A$.
\end{proof}
\begin{prop}
\label{prop:mubound}In the no singleton case, $\mu\geq x(x+a)\one$;
and the number of chains of the form $\emptyset,\{d\},\{d,o\},\ldots$
with $d\in X,o\in[n]\setminus X\setminus A$, which do not contain
any set from $\X$, is $\bigl(\mu-x(x+a)\one\bigr)n!$.
\end{prop}

\begin{proof}
A total of $\mu n!$ chains start with an element of $X$ and do not
contain any set from $\X$. The chains of the form $\emptyset,\{d_{1}\},\{d_{1},d_{2}\},\ldots$
with $d_{1}\in X,d_{2}\in X\cup A$ never contain a set from $\X$
when $\X$ contains no singleton. The number of these chains is $xn\cdot(xn+an-1)\cdot(n-2)!=\bigl(x(x+a)\one\bigr)n!$.
For the rest, the second element added to form the chain is from $[n]\setminus X\setminus A$.
\end{proof}
\begin{nameddef}[Notation]
Let $X'=X\cup A$. Let $\Y=\bigl\{\{d,o\}:d\in X,o\in[n]\setminus X\setminus A\bigr\}$,
and let $\Z=\bigl\{\{e,o\}:e\in A\setminus\tilde{A},o\in[n]\setminus X\setminus A\bigr\}\bigr\}$.
In the \emph{singletons case}, let $\X'=\Y\sqcup\B\sqcup\Z$. (Note
that here and in the rest of the paper, $\sqcup$ stands for a union
of sets which are pairwise disjoint.) In the \emph{no singleton case},
let $\X'=\X\sqcup\B\sqcup\Z$.
\end{nameddef}

\begin{prop}
The three families which make up $\X'$ are indeed disjoint in each
case, and their union forms an antichain.
\end{prop}

\begin{proof}
$\B$ is an antichain by Proposition \ref{prop:Bform}; $\X$ is an
antichain by definition; and $\Y$ and $\Z$ are antichains because
both consist of size 2 sets only. Let $D=\{d,o_{1},\ldots,o_{k}\}\in\X$,
$Y=\{d,o\}\in\Y$, $B=\{e_{1},p_{1},\ldots,p_{l}\}\in\B$ and $Z=\{e_{2},q\}\in\Z$
with $d\in X$, $e_{1}\in\tilde{A}$, $e_{2}\in A\setminus\tilde{A}$,
$o_{i},o,p_{i},q\in[n]\setminus X\setminus A$, and $l\geq1$. $B$
is unrelated to $D$ by definition, and to $Y$ because $d\notin B$
and $\left|B\right|\geq2$. $Z$ is unrelated to $D$ and $Y$ because
$d\notin Z$ and $e_{2}\notin D,Y$; $Z$ is unrelated to $B$ because
$e_{1}\notin Z$ and $e_{2}\notin B$.
\end{proof}
\begin{prop}
\label{prop:indcondition}Sets in $\C$ are disjoint from $X'$, and
they are unrelated to every set in $\X'$ (in both cases).
\end{prop}

\begin{proof}
For every $C\in\C$, $C\subseteq[n]\setminus X'$ and it is unrelated
to every set in $\X$ by definition. $C$ is unrelated to every set
in $\B$ by Proposition \ref{prop:Bform}. It also cannot be a superset
of a $Y\in\Y$ or a $Z\in\Z$, since those contain an element of $X$
or $A$; neither a proper subset of $Y$ or $Z$ because $\left|Y\right|=\left|Z\right|=2\leq\left|C\right|$.
\end{proof}
\begin{prop}
\label{prop:indchainbound}The number of chains that start with an
element of $[n]\setminus X'$ and contain a set from $\X'$ is

\begin{itemize}
\item at least $\left[x(1-x-a)\One+\left(\beta-\tilde{a}(1-x-a)\One+\nu\right)+(1-x-a)\left(a-\tilde{a}\right)\One\right]n!$
in the \emph{singletons case}, and
\item at least $\left[\left(\alpha-x+\mu\right)+\left(\beta-\tilde{a}(1-x-a)\One+\nu\right)+(1-x-a)\left(a-\tilde{a}\right)\One\right]n!$
in the \emph{no singleton case}.
\end{itemize}
\end{prop}

\begin{proof}
The number of chains on which the singleton is $\{o\}$ with $o\in[n]\setminus X'=[n]\setminus X\setminus A$,
and the second set is $\{o,d\}\in\Y$ with $d\in X$, is $xn\cdot(1-x-a)n\cdot(n-2)!=x(1-x-a)\One n!$.
The number of chains on which the singleton is $\{o\}$, and the second
set is $\{o,e\}\in\Z$ with $e\in A\setminus\tilde{A}$, is $(1-x-a)n\cdot(a-\tilde{a})n\cdot(n-2)!=(1-x-a)\left(a-\tilde{a}\right)\One n!$.
The rest follows from Proposition \ref{prop:Xchainbound} and Proposition
\ref{prop:Bchainbound}.
\end{proof}
\begin{prop}
\label{prop:indmubound}The number of chains on which the singleton
is $\{o\}$ with $o\in[n]\setminus X'$, the second set is $\{o,d\}$
with $d\in X'=X\cup A$, and which do not contain any set from $\X'$,
is

\begin{itemize}
\item $\nu n!$ in the \emph{singletons case}, and
\item $\bigl(\mu-x(x+a)\one+\nu\bigr)n!$ in the \emph{no singleton case}.
\end{itemize}
\end{prop}

\begin{proof}
Let $\A=\emptyset,A_{1},A_{2},\ldots,A_{n-1},[n]$ be a chain with
$\emptyset\subset A_{1}\subset A_{2}\subset\ldots\subset A_{n-1}\subset[n]$.
Let $\varphi(\A)$ be the chain $\emptyset,A_{2}\setminus A_{1},A_{2},A_{3},\ldots,A_{n-1},[n]$.
(In other words, in the order in which elements of $[n]$ are added
to form the chain, the first two are swapped.) $\varphi$ is a bijection.

It is easy to check that $\X'$ does not contain singletons. $\varphi$
is a bijection between chains of the form $\emptyset,\{o\},\{o,d\},\ldots$
containing no set from $\X'$, and chains of the form $\emptyset,\{d\},\{o,d\},\ldots$
containing no set from $\X'$, with $o\in[n]\setminus X'$ and $d\in X\cup A$.
Below we classify the chains $\emptyset,\{d\},\{o,d\},\ldots$ based
on what set $d$ belongs to and count them separately.

\begin{itemize}
\item For $d\in X$, $\{o,d\}\in\Y$ in the singletons case. In the no singleton
case, $\bigl(\mu-x(x+a)\one\bigr)n!$ chains of the form $\emptyset,\{d\},\{o,d\},\ldots$
contain no set from $\X$ by Proposition \ref{prop:mubound}; these
chains also contain no set from $\B$ or $\Z$, since sets from those
do not contain any element of $X$.
\item For $d\in\tilde{A}$, the number of chains of this form which contain
no set from $\B$ is $\nu n!$; these chains also contain no set from
$\X,\Y$ or $\Z$, since sets from those contain no element of $\tilde{A}$.
\item For $d\in A\setminus\tilde{A}$, $\{o,d\}\in\Z$.
\end{itemize}
Summing these cases, we get the statement of the proposition.
\end{proof}

\subsection{\label{subsec:Induction}Inductive step}
\begin{nameddef}[Notation]
Using standard notation for intervals, let $\left[A,[n]\right]$
denote the Boolean lattice $\left\{ S\subseteq[n]:A\subseteq S\right\} $.
Let $[n]\setminus X'=[n]\setminus X\setminus A=\{o_{1},o_{2},\ldots,o_{(1-x-a)n}\}$;
and for a family of sets $\A$, let $\A-o_{i}=\left\{ S\setminus\{o_{i}\}:S\in\A\right\} $.
Let $\C'_{i}=\left(\C\cap\bigl[\{o_{i}\},[n]\bigr]\right)-o_{i}$,
and $\X'_{i}=\left(\X'\cap\bigl[\{o_{i}\},[n]\bigr]\right)-o_{i}$.
Let $\alpha'_{i}=l(n-1,\X'_{i})$. (Here the Lubell function on the
Boolean lattice $2^{[n]\setminus\{o_{i}\}}$ of order $n-1$ is used.)
\end{nameddef}

$\C'_{i}\subseteq2^{[n]\setminus\{o_{i}\}}$ is a $\Lambda$-free
family which does not contain $\emptyset$ (since $o_{i}\notin A$,
so $\{o_{i}\}\notin\F$), nor any set larger than $n-1-n'$. Sets
in $\C'_{i}$ are disjoint from $X'$, and are unrelated to sets in
$\X'_{i}$ by Proposition \ref{prop:indcondition}. Moreover, every
set in $\X'_{i}$ contains exactly one element of $X'$. Therefore,
the conditions of Lemma \ref{lem:ind} are satisfied for the family
$\C'_{i}\subseteq2^{[n]\setminus\{o_{i}\}}$ where the corresponding
``forbidden'' set is $X'\subseteq[n]\setminus\{o_{i}\}$, with $\frac{\left|X'\right|}{n-1}=(x+a)\One$
and the corresponding ``forbidden'' antichain is $\X'_{i}$.

Since $\X'_{i}$ is an antichain, $\alpha'_{i}(n-1)!$ is the number
of chains in $2^{[n]\setminus\{o_{i}\}}$ that contain a set from
$\X'_{i}$. Chains of $2^{[n]\setminus\{o_{i}\}}$ correspond to chains
of $2^{[n]}$ that start with $\{o_{i}\}$. So by Proposition \ref{prop:indchainbound},
in the \emph{singletons case} 
\[
\sum_{i=1}^{(1-x-a)n}\alpha'_{i}\geq\left[x(1-x-a)\One+\left(\beta-\tilde{a}(1-x-a)\One+\nu\right)+(1-x-a)\left(a-\tilde{a}\right)\One\right]n,
\]
 and in the \emph{no singleton case} 
\[
\sum_{i=1}^{(1-x-a)n}\alpha'_{i}\geq\left[\left(\alpha-x+\mu\right)+\left(\beta-\tilde{a}(1-x-a)\One+\nu\right)+(1-x-a)\left(a-\tilde{a}\right)\One\right]n.
\]

Let $\mu'_{i}(n-1)!$ be the number of chains in the Boolean lattice
$2^{[n]\setminus\{o_{i}\}}$ which start with an element of $X'$
as a singleton, but do not contain any set from $\X'_{i}$. By Proposition
\ref{prop:indmubound}, in the \emph{singletons case }
\[
\sum_{i=1}^{(1-x-a)n}\mu'_{i}=\nu n,
\]
 and in the \emph{no singleton case} 
\[
\sum_{i=1}^{(1-x-a)n}\mu'_{i}=\bigl(\mu-x(x+a)\one+\nu\bigr)n.
\]

Let $c'_{i}(n-1)!$ be the number of MNM chains w.r.t.\ $\C'_{i}$
in $2^{[n]\setminus\{o_{i}\}}$. The corresponding $2^{[n]}$-chains,
starting with $\{o_{i}\}$, are MNM chains w.r.t.\ $\F$. The total
number of MNM chains w.r.t.\ $\F$ is $cn!$, out of which $c_{0}n!$
start with an element of $A$ as a singleton. By Proposition \ref{prop:cbound},
\[
\sum_{i=1}^{(1-x-a)n}c'_{i}=(c-c_{0})n=\left(c-\tilde{a}(x+a)\one-\nu\right)n.
\]

The following two examples are typical cases where, in the induction
step for the $\C'_{i}$'s, we will get the singletons case and the
no singleton case respectively.
\begin{example}
Let $X=\X=\emptyset$ and $\B=\bigl\{\{e,o\}:e\in A,o\in[n]\setminus A\bigr\}$.
Then $\tilde{A}=A$, $\beta=2a(1-a)\One$, and $\nu=0$. $X'=A$,
and $\X'_{i}=\bigl\{\{e\}:e\in A\bigr\}$. $\sum_{i=1}^{(1-a)n}\alpha'_{i}=a(1-a)\One n$,
$\alpha'_{i}=a\One=\frac{\left|X'\right|}{n-1}$ and $\mu'_{i}=0$.
$\sum_{i=1}^{(1-x-a)n}c'_{i}=\left(c-a^{2}\right)n$ and the average
of the $c'_{i}$'s is $\frac{c-a^{2}}{1-a}$.
\end{example}

\begin{example}
Let $X=\X=\emptyset$ and $\B\subseteq\hat{\B}:=\bigl\{\{e,o_{1},o_{2}\}:e\in A,o_{1},o_{2}\in[n]\setminus A\bigr\}$.
Then $X'=A$, and $\X'_{i}\subseteq\bigl\{\{e,o\}:e\in A,o\in[n]\setminus A\setminus\{o_{i}\}\bigr\}$.
Chains on $2^{[n]}$ of the form $\emptyset,\{e_{1}\},\{e_{1},o\},\{e_{1},o,e_{2}\},\ldots$
do not intersect $\B$. So $\sum_{i=1}^{(1-a)n}\mu'_{i}=\nu\geq a^{2}(1-a)\One^{2}\frac{\One a(n-1)-1}{\One a(n-2)}n$
(greater if $\B\subsetneqq\hat{\B}$), and the average of the $\mu'_{i}$'s
is $\geq a^{2}\One^{2}\frac{\One a(n-1)-1}{\One a(n-2)}={x'}^{2}\frac{x'(n-1)-1}{x'((n-1)-1)}$
where $x'=\frac{\left|X'\right|}{n-1}$. In the case of $\B=\hat{\B}$,
the size of the sets in $\C$ is at least 3, and the size of those
in $\C'_{i}$ is at least 2.

\end{example}

\begin{prop}
\label{prop:lubell_induction}
\[
l(\C)=\frac{1}{n}\sum_{i=1}^{(1-x-a)n}l(n-1,\C'_{i})\quad\textnormal{and}\quad l(\F)=a+\beta+l(\C)=a+\beta+\frac{1}{n}\sum_{i=1}^{(1-x-a)n}l(n-1,\C'_{i}).
\]
 (Still understanding the one parameter version $l(\F)$ as $l(n,\F)$
for a family $\F\subseteq2^{[n]}$.)
\end{prop}

\begin{proof}
Every chain in the Boolean lattice $2^{[n]}$ that intersects $\C$
has an $\{o_{i}\}$ as a singleton, and thus corresponds to a chain
in the Boolean lattice $\left[\{o_{i}\},[n]\right]-o_{i}$ that intersects
$\C'_{i}$. 
\[
l(\C)=\frac{1}{n!}\sum_{\h\text{ is a chain in }2^{[n]}}\left|\h\cap\C\right|=\frac{1}{n!}\sum_{i=1}^{(1-x-a)n}\sum_{\h\text{ is a chain in }\left[\{o_{i}\},[n]\right]-o_{i}}\left|\h\cap\C'_{i}\right|=\frac{1}{n}\sum_{i=1}^{(1-x-a)n}l(n-1,\C'_{i}).
\]

Let $\A=\bigl\{\{e\}:e\in A\bigr\}$. Then $\F=\A\sqcup\B\sqcup\C$.
So $l(\F)=l(\A)+l(\B)+l(\C)$ with $l(\A)=\frac{\left|A\right|}{n}=a$
and $l(\B)=\beta$.
\end{proof}
We now prove Lemma \ref{lem:ind} (and thus Lemma \ref{lem:lubell})
using induction on $n$. According to Proposition \ref{prop:indbase},
Lemma \ref{lem:ind} holds for $n\leq n'$. By induction and Lemma
\ref{lem:functions} Point \ref{enu:concave}, 
\begin{align*}
l(n-1,\C'_{i}) & \leq f\!\left((x+a)\One,c'_{i}+\mu'_{i}+\frac{1}{n'}\right)-\left(\alpha'_{i}-\mu'_{i}-(x+a)\One\right)+\frac{3}{n'}\\
 & \leq f\!\left(x+a,c'_{i}+\mu'_{i}+\frac{1}{n'}\right)-\left(\alpha'_{i}-\mu'_{i}-(x+a)\One\right)+\frac{3}{n'}.
\end{align*}
So, by Proposition \ref{prop:lubell_induction}, we have
\begin{align*}
l(\C) & =\frac{1}{n}\sum_{i=1}^{(1-x-a)n}l(n-1,\C'_{i})\leq\frac{1}{n}\sum_{i=1}^{(1-x-a)n}f\!\left(x+a,c'_{i}+\mu'_{i}+\frac{1}{n'}\right)\\
 & {}-\frac{1}{n}\left(\sum_{i=1}^{(1-x-a)n}\alpha'_{i}-\sum_{i=1}^{(1-x-a)n}\mu'_{i}-\sum_{i=1}^{(1-x-a)n}(x+a)\One\right)+\frac{1}{n}\cdot\frac{3(1-x-a)n}{n'}.
\end{align*}

We handle the case of $1-x-a=0$ separately. If $1-x-a=0$, $A=[n]\setminus X$
and, since any non-singleton $\{e_{1},e_{2},\ldots\}\in\F$ would
form a $\Lambda$ with the singletons $\{e_{1}\},\{e_{2}\}\in\F$,
we have $\F=\A$ and $l(\F)=a=1-x$. This is only possible in the
singletons case, since a non-singleton in $\X$ would have to contain
elements of $[n]\setminus X\setminus A$. In the singletons case $\alpha=x$
and $\mu=0$, so $l(\F)=1-x\leq f(x,c)\leq f\!\left(x,c+\mu+\frac{1}{n'}\right)-(\alpha-\mu-x)+\frac{3}{n'}$
by Lemma \ref{lem:functions} Point \ref{enu:1-x}. From now on, we
assume that $1-x-a>0$.

Since $f$ is concave in $c$, by Jensen's inequality, and since $f$
is monotonously decreasing in $x$, 
\begin{align*}
l(\C) & \leq(1-x-a)f\!\left(x+a,\frac{\sum_{i=1}^{(1-x-a)n}c'_{i}+\sum_{i=1}^{(1-x-a)n}\mu'_{i}}{(1-x-a)n}+\frac{1}{n'}\right)\\
 & {}-\left(\frac{1}{n}\sum_{i=1}^{(1-x-a)n}\alpha'_{i}-\frac{1}{n}\sum_{i=1}^{(1-x-a)n}\mu'_{i}-(1-x-a)(x+a)\One\right)+\frac{3(1-x-a)}{n'}.
\end{align*}

\emph{Correction term calculations} that we will use later (assuming
$n'\leq n-1$):\emph{ }
\begin{equation}
\begin{gathered}(1-\one)(x+\tilde{a})(x+a)+\frac{1-x-a}{n'}=\frac{(x+\tilde{a})(1-x-a)}{n-1}+\frac{1-x-a}{n'}\\
\leq\frac{(1+x+\tilde{a})(1-x-a)}{n'}\leq\frac{1-(x+a)^{2}}{n'}\leq\frac{1}{n'}.
\end{gathered}
\label{eq:correction_nosing_c}
\end{equation}
\emph{
\begin{equation}
(1-\one)\tilde{a}(x+a)+\frac{1-x-a}{n'}\leq(1-\one)(x+\tilde{a})(x+a)+\frac{1-x-a}{n'}\leq\frac{1}{n'}.\label{eq:correction_singletons_c}
\end{equation}
\begin{equation}
\begin{gathered}2\tilde{a}(1-x-a)(\One-1)+2(\One-1)x-\left(2(\One-1)+(\one-1)\right)x(x+a)+\frac{3(1-x-a)}{n'}\\
\leq\frac{2a+3x}{n-1}+\frac{3(1-x-a)}{n'}\leq\frac{3}{n'}.
\end{gathered}
\label{eq:correction_nosing}
\end{equation}
\begin{equation}
2\tilde{a}(1-x-a)(\One-1)+\frac{3(1-x-a)}{n'}\leq\frac{2a}{n-1}+\frac{3(1-x-a)}{n'}\leq\frac{3}{n'}.\label{eq:correction_singletons}
\end{equation}
}

\emph{In the singletons case:} 
\begin{align*}
l(\F) & \leq a+\beta+(1-x-a)f\!\left(x+a,\frac{\left(c-\tilde{a}(x+a)\one-\nu\right)+\nu}{1-x-a}+\frac{1}{n'}\right)\\
 & {}-\Bigl(\left[x(1-x-a)\One+\left(\beta-\tilde{a}(1-x-a)\One+\nu\right)+(1-x-a)\left(a-\tilde{a}\right)\One\right]\\
 & {}-\nu-(1-x-a)(x+a)\One\Bigr)+\frac{3(1-x-a)}{n'}\displaybreak[0]\\
 & =a+(1-x-a)f\!\left(x+a,\frac{c-\tilde{a}(x+a)\one}{1-x-a}+\frac{1}{n'}\right)+2\tilde{a}(1-x-a)\One+\frac{3(1-x-a)}{n'}\displaybreak[0]\\
 & =a+(1-x-a)f\!\left(x+a,\frac{c-\tilde{a}(x+a)+(1-\one)\tilde{a}(x+a)+\frac{1-x-a}{n'}}{1-x-a}\right)+2\tilde{a}(1-x-a)\\
 & {}+2\tilde{a}(1-x-a)(\One-1)+\frac{3(1-x-a)}{n'}.
\end{align*}
 By Lemma \ref{lem:functions} Point \ref{enu:concave}~and Point
\ref{enu:g}, and \eqref{eq:correction_singletons_c} and \eqref{eq:correction_singletons}
in the Correction term calculations (note that in this case $\alpha=x$
and $\mu=0$), 
\[
l(\F)\leq g\!\left(x,c+\frac{1}{n'},a,\tilde{a}\right)+\frac{3}{n'}\leq f\!\left(x,c+\frac{1}{n'}\right)+\frac{3}{n'}=f\!\left(x,c+\mu+\frac{1}{n'}\right)-(\alpha-\mu-x)+\frac{3}{n'}.
\]
(Note that $\tilde{a}\leq\frac{c}{(x+a)\one}$, so $0\leq\frac{c-\tilde{a}(x+a)\one}{1-x-a}\leq\frac{c+\frac{1}{n'}-\tilde{a}(x+a)}{1-x-a}$,
and $\tilde{a}\leq\frac{c+\frac{1}{n'}}{x+a}$.)

\emph{In the no singleton case:} 
\begin{align*}
l(\F) & \leq a+\beta+(1-x-a)f\!\left(x+a,\frac{\left(c-\tilde{a}(x+a)\one-\nu\right)+\mu-x(x+a)\one+\nu}{1-x-a}+\frac{1}{n'}\right)\\
 & {}-\Bigl(\left[\left(\alpha-x+\mu\mathnote{here}\right)+\left(\beta-\tilde{a}(1-x-a)\One+\nu\right)+(1-x-a)\left(a-\tilde{a}\right)\One\right]\\
 & {}-\left[\mu-x(x+a)\one+\nu\right]-(1-x-a)(x+a)\One\Bigr)+\frac{3(1-x-a)}{n'}\displaybreak[0]\\
 & =a+(1-x-a)f\!\left(x+a,\frac{c+\mu-(x+\tilde{a})(x+a)\one}{1-x-a}+\frac{1}{n'}\right)-(\alpha-x(x+a)\one\mathnote{here}-x)\\
 & {}+2\tilde{a}(1-x-a)\One+(2\cdot\One-1)x-(2\cdot\One+\one)x(x+a)+\frac{3(1-x-a)}{n'}\displaybreak[0]\\
 & =a+(1-x-a)f\!\left(x+a,\frac{c+\mu-(x+\tilde{a})(x+a)+(1-\one)(x+\tilde{a})(x+a)+\frac{1-x-a}{n'}}{1-x-a}\right)\\
 & {}+2\tilde{a}(1-x-a)+x-3x(x+a)-(\alpha-x(x+a)\one\mathnote{here}-x)\\
 & {}+2\tilde{a}(1-x-a)(\One-1)+2(\One-1)x-\left(2(\One-1)+(\one-1)\right)x(x+a)+\frac{3(1-x-a)}{n'}.
\end{align*}
 By Lemma \ref{lem:functions} Point \ref{enu:concave}~and Point
\ref{enu:h}, Proposition \ref{prop:mubound}. and \eqref{eq:correction_nosing_c}
and \eqref{eq:correction_nosing} in the Correction term calculations,
\[
l(\F)\leq h\!\left(x,c+\mu+\frac{1}{n'},a,\tilde{a}\right)-(\alpha-\mu-x)+\frac{3}{n'}\leq f\!\left(x,c+\mu+\frac{1}{n'}\right)-(\alpha-\mu-x)+\frac{3}{n'}.
\]
(Note that $\tilde{a}\leq\frac{c}{(x+a)\one}$, so $0\leq\frac{c-\tilde{a}(x+a)\one}{1-x-a}\leq\frac{c+\mu+\frac{1}{n'}-(x+\tilde{a})(x+a)}{1-x-a}$,
and $\tilde{a}\leq\frac{c+\mu+\frac{1}{n'}}{x+a}-x$.)

\section{Diamond-free families \textendash{} Proof of Theorem \ref{thm:main}\label{sec:diamond}}

Let $\F$ be a diamond-free family on $2^{[n]}$.

We cite Lemma 1 from \cite{diamond_228}:
\begin{lem}[Axenovich, Manske, Martin \cite{diamond_228}]
\[
\sum_{\substack{k\in\{0,1,\ldots,n\}\\
\left|k-n/2\right|\geq n^{2/3}
}
}\binom{n}{k}\leq2^{n-\Omega(n^{1/3})}=2^{-\Omega(n^{1/3})}\binom{n}{\left\lfloor n/2\right\rfloor }.
\]
\end{lem}

By this lemma, the number of sets in $\F$ in the top and bottom $n':=n/2-n^{\frac{2}{3}}$
levels is $o(1)\binom{n}{\left\lfloor n/2\right\rfloor }$, so, since
we are bounding the cardinality of $\F$, we may assume that those
levels do not contain any set from $\F$.
\begin{nameddef}[Notation]
For $c\in[0,1]$, let $\tilde{c}=\min\!\left(c,\frac{1}{4}\right)$,
and let $f(c)=1-\tilde{c}+\sqrt{\tilde{c}}$. (This is equal to $f(0,c)$
as defined in Definition \ref{def:functions}.) For $A\in\F$, recall
that $\left[A,[n]\right]$ denotes the Boolean lattice $\left\{ S\subseteq[n]:A\subseteq S\right\} $.
A chain of this lattice is of the form $A\subset A_{\left|A\right|+1}\subset A_{\left|A\right|+2}\subset\ldots\subset A_{n-1}\subset[n]$.
(When saying just ``chain'', we continue to mean a maximal chain
in the Boolean lattice $2^{[n]}$.) Let 
\[
c(A)=\frac{1}{\left(n-\left|A\right|\right)!}\#\left\{ \begin{gathered}\C\textnormal{ is a chain in \ensuremath{\left[A,[n]\right]}}:\\
\C\textnormal{ is MNM w.r.t.\ }\F\cap\left[A,[n]\right]
\end{gathered}
\right\} .
\]
 Further, we can assume without loss of generality that 
\[
C:=\frac{1}{n!}\#\left\{ \begin{gathered}\C\textnormal{ is a chain: \ensuremath{\C\cap\F=\emptyset} or }\\
\min(\C\cap\F)\textnormal{ is not minimal in }\F
\end{gathered}
\right\} \geq\frac{1}{n!}\#\left\{ \begin{gathered}\C\textnormal{ is a chain: \ensuremath{\C\cap\F=\emptyset} or }\\
\C\textnormal{ is MNM w.r.t.\ }\F
\end{gathered}
\right\} .
\]
\end{nameddef}

(If this does not hold, we can replace $\F$ with $\left\{ [n]\setminus A:A\in\F\right\} $:
this family is diamond-free, has the same cardinality, and the opposite
inequality holds.) Clearly $C\geq\dfrac{1}{n!}\#\left\{ \begin{gathered}\C\textnormal{ is a chain:}\\
\C\textnormal{ is MNM w.r.t.\ }\F
\end{gathered}
\right\} $.

\[
l(\F)=\frac{1}{n!}\sum_{\C\textnormal{ is a chain}}\#(\C\cap\F)=\frac{1}{n!}\sum_{A\in\F}\sum_{\substack{\C\textnormal{ is a chain}\\
A=\min(\C\cap\F)
}
}\#(\C\cap\F),
\]
 since each chain $\C$ will be counted when $A=\min(\C\cap\F)$ \textendash{}
except if $\C\cap\F=\emptyset$, but then $\#(\C\cap\F)=0$. Continuing,
\[
l(\F)=\frac{1}{n!}\sum_{\substack{A\in\F\\
\exists\textnormal{ a chain \ensuremath{\C}: }A=\min(\C\cap\F)
}
}\#\left\{ \begin{gathered}\C\textnormal{ is a chain containing }A:\\
A=\min(\C\cap\F)
\end{gathered}
\right\} \frac{{\displaystyle \sum_{\substack{\C\textnormal{ is a chain}\\
A=\min(\C\cap\F)
}
}\#(\C\cap\F})}{\#\left\{ \begin{gathered}\C\textnormal{ is a chain containing }A:\\
A=\min(\C\cap\F)
\end{gathered}
\right\} }.
\]
 Each chain on $\left[A,[n]\right]$ can be extended to a full $2^{[n]}$-chain
in $\left|A\right|!$ ways. Furthermore, the Boolean lattice $\left[A,[n]\right]$
can be made equivalent to the Boolean lattice $2^{[n]\setminus A}$
by subtracting $A$ from each set; for $\A\subseteq\left[A,[n]\right]$,
we denote $\A-A=\left\{ S\setminus A:S\in\A\right\} $. If $A=\min(\C\cap\F)$,
$\#(\C\cap\F)=\#(\C\cap\left[A,[n]\right]\cap\F)$. If $A$ is minimal
in $\F$ (that is, on every chain), 
\begin{gather*}
\frac{{\displaystyle \sum_{\substack{\C\textnormal{ is a chain}\\
A=\min(\C\cap\F)
}
}\#(\C\cap\F)}}{\#\left\{ \begin{gathered}\C\textnormal{ is a chain containing }A:\\
A=\min(\C\cap\F)
\end{gathered}
\right\} }=\frac{\left|A\right|!{\displaystyle \sum_{\C\textnormal{ is a chain in }\left[A,[n]\right]}\#(\C\cap\F\cap\left[A,[n]\right])}}{\left|A\right|!\left(n-\left|A\right|\right)!}\\
=\frac{{\displaystyle \sum_{\C\textnormal{ is a chain in }2^{[n]\setminus A}}\#(\C\cap((\F\cap\left[A,[n]\right])-A))}}{\left(n-\left|A\right|\right)!}=l(n-\left|A\right|,(\F\cap\left[A,[n]\right])-A).
\end{gather*}
 $(\F\cap\left[A,[n]\right])-A$ is diamond-free, so $((\F\cap\left[A,[n]\right])-A)\setminus\emptyset$
is $\Lambda$-free; and the top $n'$ levels are assumed to be empty.
Using Lemma \ref{lem:lubell} as well as that $\frac{1}{n'}=\frac{1}{\Omega(n)}=o(1)$
and the subadditivity of the square root function,
\begin{gather*}
l\bigl(n-\left|A\right|,\bigl((\F\cap\left[A,[n]\right])-A\bigr)\setminus\emptyset\bigr)\leq1-\min\!\left(c(A)+\frac{1}{n'},\frac{1}{4}\right)+\sqrt{\min\!\left(c(A)+\frac{1}{n'},\frac{1}{4}\right)}+\frac{3}{n'}\\
\leq f(c(A))+\sqrt{\frac{1}{n'}}+\frac{3}{n'}=f(c(A))+o(1),
\end{gather*}
so $l\bigl(n-\left|A\right|,(\F\cap\left[A,[n]\right])-A\bigr)\leq1+f(c(A))+o(1)$.
Whereas if $A$ is not minimal in $\F$, i.e. $\exists S\in\F$ such
that $A\supsetneqq S$, then for any chain $\C$ for which $\min(\C\cap\F)=A$,
we have $\#(\C\cap\F)\leq2$ (otherwise $S$ and three sets in $\C\cap\F$
would form a diamond), so $\frac{\sum_{\substack{\C\textnormal{ is a chain}\\
A=\min(\C\cap\F)
}
}\#(\C\cap\F)}{\#\left\{ \substack{\C\textnormal{ is a chain through }A:\\
A=\min(\C\cap\F)
}
\right\} }\leq2$.

\begin{align*}
l(\F) & \leq\frac{1}{n!}\sum_{\substack{A\in\F\\
A\textnormal{ is minimal in }\F
}
}\#\left\{ \C\textnormal{ is a chain containing }A\right\} \left(1+f(c(A))+o(1)\right)\\
 & {}+\frac{1}{n!}\sum_{\substack{A\in\F\\
A\textnormal{ is not minimal in }\F
}
}\#\left\{ \begin{gathered}\C\textnormal{ is a chain containing }A:\\
A=\min(\C\cap\F)
\end{gathered}
\right\} \cdot2\\
 & \le2+\frac{1}{n!}\sum_{\substack{A\in\F\\
A\textnormal{ is minimal in }\F
}
}\#\left\{ \C\textnormal{ is a chain containing }A\right\} \left(f\left(c\left(A\right)\right)-1+o(1)\right).
\end{align*}
 Since $f$ is concave, we can use Jensen's inequality with the weights
$\frac{\#\left\{ \C\textnormal{ is a chain containing }A\right\} }{(1-C)n!}$
(where $A$ is minimal in $\F$). Notice that the sum of all the weights
is $1$ because the sum of numerators is the total number of chains
$\C$ where $\min(\C\cap\F)$ is minimal in $\F$, that is, $(1-C)n!$.
\begin{gather*}
l(\F)\leq2+(1-C)\left(f\!\left(\frac{{\displaystyle \sum_{\substack{A\in\F\\
A\textnormal{ is minimal in }\F
}
}\#\left\{ \C\textnormal{ is a chain containing }A\right\} c(A)}}{(1-C)n!}\right)-1+o(1)\right).
\end{gather*}
 $c(A)$ is the fraction of the chains containing $A$ which are MNM,
so $\#\left\{ \C\textnormal{ is a chain containing }A\right\} c(A)$
is the number of MNM chains through $A$. In the numerator, each MNM
chain in the whole Boolean lattice is counted once, except if the
minimal element on it is not a global minimal, then it is not counted.
So the numerator is less than or equal to the total number of MNM
chains in the Boolean lattice, which is at most $Cn!$. Substituting,
we get 
\[
l(\F)\leq2+(1-C)\left(f\!\left(\frac{n!C}{n!(1-C)}\right)-1+o(1)\right)=1+C+(1-C)f\!\left(\frac{C}{1-C}\right)+o(1).
\]
 $C$ varies between 0 and 1. $\frac{C}{1-C}$ is increasing in $C$.
Above $\frac{C}{1-C}=\frac{1}{4}$ (corresponding to $C=\frac{1}{5}$),
$f(\frac{C}{1-C})$ is constant $\frac{5}{4}$, so $1+C+(1-C)f\!\left(\frac{C}{1-C}\right)=\frac{9}{4}-\frac{C}{4}$
is decreasing in C. So it is enough to take the maximum in the interval
$\left[0,\frac{1}{5}\right]$: 
\begin{align*}
\frac{\left|\F\right|}{\binom{n}{\left\lfloor n/2\right\rfloor }} & \leq l(\F)\leq\max_{C\in\left[0,\frac{1}{5}\right]}\left(1+C+(1-C)\left(1-\frac{C}{1-C}+\sqrt{\frac{C}{1-C}}\right)\right)+o(1)\\
 & =\frac{\sqrt{2}+3}{2}+o(1)<2.20711+o(1).
\end{align*}

\section*{Acknowledgments\phantomsection\addcontentsline{toc}{section}{Acknowledgements}}

We thank the anonymous referees for their detailed comments which
helped improve the presentation of our paper. We would also like to
thank D\"om\"ot\"or P\'alv\"olgyi for helpful discussions concerning
the $3$-level version of this problem. The research of the second
and third authors was supported by the National Research, Development
and Innovation Office \textendash{} NKFIH, grant K 116769.

\bibliographystyle{plain}
\phantomsection\addcontentsline{toc}{section}{\refname}\bibliography{bibliography1}

\newcounter{equationsave}
\setcounter{equationsave}{\value{equation}}

\appendix
\counterwithout{equation}{section}
\setcounter{equation}{\value{equationsave}}\renewcommand\thesection{Appendix \Alph{section}}

\section{Proof of Lemma \ref{lem:functions}}

\makeatletter\def\@currentlabel{\Alph{section}}\makeatother \label{sec:functions_proof}
Points \ref{enu:x=00003D0} and \ref{enu:1-x} are easy to see.
\begin{boldproof}[Proof of Point \ref{enu:concave}.]
It is also easy to check that $f$ is continuous at the points $x=\frac{1}{2}$,
$c=4\left(x-x^{2}\right)^{2}$, $c=\frac{1}{4}$, and that the function
is monotonously decreasing in $x$ and increasing in $c$ in each
range.

$f(x,0)=1-x$; $x^{2}-2x+1-c+\sqrt{c}$ is a concave and monotonously
increasing expression in $c$. When $0<x\leq\frac{1}{2}$, $c\mapsto1-x+\left(\frac{1}{4\left(x-x^{2}\right)}-1\right)c$
is the tangential line of the graph of $c\mapsto x^{2}-2x+1-c+\sqrt{c}$
at the point $c=4\left(x-x^{2}\right)^{2}$, since both their values,
and their derivatives at this point coincide. So $f$ is concave in
$c$.
\end{boldproof}
Since the graph of a concave function is below the tangent line at
any point, we also have that for $x\in\left[0,\frac{1}{2}\right],c\in\left[0,\frac{1}{4}\right]$,
\begin{equation}
f(x,c)\geq x^{2}-2x+1-c+\sqrt{c}=:\tilde{f}(x,c).\label{eq:fbound}
\end{equation}
We will use this inequality in the proof of Point \ref{enu:g}\ and
\ref{enu:h}.

\begin{boldproof}[Proof of Point \ref{enu:g}.]
If $x+a=0$, $g(x,c,a,\tilde{a})=f(x,c)$. From now on, we assume
that $x+a>0.$

We first show that $g$ is monotonously increasing in $\tilde{a}$.
\begin{equation}
\begin{gathered}\left(\frac{\partial}{\partial c}f\right)\!(x,c)=\left(\begin{cases}
\frac{1}{4\left(x-x^{2}\right)}-1 & \tif x\leq\frac{1}{2}\tand c<4\left(x-x^{2}\right)^{2}\\
-1+\frac{1}{2\sqrt{c}} & \tif x\leq\frac{1}{2}\tand4\left(x-x^{2}\right)^{2}\leq c\leq\frac{1}{4}\\
0 & \tif\frac{1}{4}\leq c\tor\frac{1}{2}\leq x
\end{cases}\right)\leq\begin{cases}
\frac{1}{4\left(x-x^{2}\right)}-1 & \tif x\leq\frac{1}{2}\\
0 & \tif\frac{1}{2}\leq x
\end{cases}\end{gathered}
\label{eq:particalcforpoint4}
\end{equation}
So, 
\begin{gather*}
\frac{\partial}{\partial\tilde{a}}g\left(x,c,a,\tilde{a}\right)=2(1-x-a)+(1-x-a)\cdot\left(\frac{\partial}{\partial c}f\right)\!\left(x+a,\frac{c-\tilde{a}(x+a)}{1-x-a}\right)\cdot\frac{\partial}{\partial\tilde{a}}\!\left(\frac{c-\tilde{a}(x+a)}{1-x-a}\right)\\
\geq2(1-x-a)-(x+a)\left(\begin{cases}
\frac{1}{4\left((x+a)-(x+a)^{2}\right)}-1 & \tif x+a\leq\frac{1}{2}\\
0 & \tif\frac{1}{2}\leq x+a
\end{cases}\right)\\
\begin{gathered}\geq2(1-x-a)-(x+a)\left(\begin{cases}
\frac{1}{4\cdot\frac{1}{2}(x+a)}-1 & \tif x+a\leq\frac{1}{2}\\
0 & \tif\frac{1}{2}\leq x+a
\end{cases}\right)\geq2(1-x-a)-\left(\begin{cases}
\frac{1}{2} & \tif x+a\leq\frac{1}{2}\\
0 & \tif\frac{1}{2}\leq x+a
\end{cases}\right)\geq0.\end{gathered}
\end{gather*}
 Therefore, from now on we assume $\tilde{a}=\min\!\left(a,\frac{c}{x+a}\right)$
since if $g(x,c,a,\tilde{a})\le f(x,c)$ holds for $\tilde{a}=\min\!\left(a,\frac{c}{x+a}\right)$
then it also holds for any $\tilde{a}\in\left[0,\min\!\left(a,\frac{c}{x+a}\right)\right].$
\gdef\labelwidthi{\widthof{\textbf{Case 0. }}}
\gdef\labelwidthii{\widthof{\textbf{Case 0.0. }}}
\gdef\labelwidthiii{\widthof{\textbf{Case 0.0.0. }}}
\gdef\labelwidthiv{\widthof{\textbf{Case 0.0.0.0. }}}
\gdef\labeli{\textbf{Case \arabic{enumi}. }}
\gdef\labelii{\textbf{Case \arabic{enumi}.\arabic{enumii}. }}
\gdef\labeliii{\textbf{Case \arabic{enumi}.\arabic{enumii}.\arabic{enumiii}. }}
\gdef\labeliv{\textbf{Case \arabic{enumi}.\arabic{enumii}.\arabic{enumiii}.\arabic{enumiv}. }}
\gdef\refi{\arabic{enumi}.}
\gdef\refii{\arabic{enumi}.\arabic{enumii}.}
\gdef\refiii{\arabic{enumi}.\arabic{enumii}.\arabic{enumiii}.}
\gdef\refiv{\arabic{enumi}.\arabic{enumii}.\arabic{enumiii}.\arabic{enumiv}.}

\begin{enumerate}[label=\labeli, ref=\refi, labelsep=0em, leftmargin=0em, labelwidth=\labelwidthi, itemindent=\labelwidth, align=left]
\item \label{case:g:a=00003Datilde}First assume $a\leq\frac{c}{x+a}$
(so \foreignlanguage{english}{$\tilde{a}=a$}), which is equivalent
to $a(x+a)\leq c$ or $a\leq\frac{-x+\sqrt{x^{2}+4c}}{2}$. Let $x'=x+a$
and $c'=\frac{c-a(x+a)}{1-x-a}$.

\begin{enumerate}[label=\labelii, ref=\refii, labelsep=0em, leftmargin=1em, labelwidth=\labelwidthii, itemindent=\labelwidth, align=left]
\item \label{case:g:smallcprime}When $x'\leq\frac{1}{2}$ and $4{\bigl(x'-{x'}^{2}\bigr)}^{2}\leq c'$,
we bound $g$ from above: 
\begin{align*}
g(x,c,a,a) & =a+(1-x-a)f\!\left(x+a,\frac{c-a(x+a)}{1-x-a}\right)+2a(1-x-a)\\
 & \leq a+(1-x-a)f\!\left(x+a,\frac{c}{1-x-a}\right)+2a(1-x-a)=:\tilde{g}(x,c,a,a).
\end{align*}
 We now consider subcases based on the values of $c$ and $\frac{c}{1-x-a}$
compared to $\frac{1}{4}$. Note that $4{\bigl(x'-{x'}^{2}\bigr)}^{2}\leq c'\leq\frac{c}{1-x-a}$.

\begin{enumerate}[label=\labeliii, ref=\refiii, labelsep=0em, leftmargin=1em, labelwidth=\labelwidthiii, itemindent=\labelwidth, align=left]
\item \label{case:g:smallc}When $c\leq\frac{c}{1-x-a}\leq\frac{1}{4}$,
using \eqref{eq:fbound}, 
\begin{equation}
\begin{gathered}g(x,c,a,a)-f(x,c)\leq\tilde{g}(x,c,a,a)-\tilde{f}(x,c)=a+(1-x-a)\biggl((x+a)^{2}-2(x+a)+1\\
{}-\frac{c}{1-x-a}+\frac{\sqrt{c}}{\sqrt{1-x-a}}\biggr)+2a(1-x-a)-\left(x^{2}-2x+1-c+\sqrt{c}\right)\\
=-x(1-x-2a)+(1-x-a)(x+a)^{2}+\left(\sqrt{1-x'}-1\right)\sqrt{c}.
\end{gathered}
\label{eq:g^-f^}
\end{equation}
Thus, 
\[
\frac{\partial}{\partial c}\!\left(\tilde{g}(x,c,a,a)-\tilde{f}(x,c)\right)=\frac{\sqrt{1-x-a}-1}{2\sqrt{c}}\leq0.
\]
 So it is enough to check that $\tilde{g}(x,c,a,a)-\tilde{f}(x,c)\leq0$
when $c'=4{\bigl(x'-{x'}^{2}\bigr)}^{2}$ or, equivalently, when $c=4{\bigl(x'-{x'}^{2}\bigr)}^{2}\left(1-x'\right)+a(x+a)$;
then it is also $\leq0$ for bigger $c$. First some auxiliary calculations:
\begin{equation}
\sqrt{1-x'}-1\leq\sqrt{1-x'+\frac{{x'}^{2}}{4}}-1=-\frac{x'}{2}\leq0.\label{eq:aux:sqrtc coeff}
\end{equation}
\[
4{\bigl(x'-{x'}^{2}\bigr)}^{2}\leq4{x'}^{2}\quad\textnormal{and}\quad4{\bigl(x'-{x'}^{2}\bigr)}^{2}\leq4\cdot\left(\frac{1}{4}\right)^{2}=\frac{1}{4},\quad\textnormal{so}
\]
\begin{equation}
\begin{gathered}4{\bigl(x'-{x'}^{2}\bigr)}^{2}\leq\min\!\left(4{x'}^{2},\frac{1}{4}\right)\leq x'\end{gathered}
.\label{eq:aux_x'}
\end{equation}
So,
\begin{equation}
\begin{gathered}4{\bigl(x'-{x'}^{2}\bigr)}^{2}\left(1-x'\right)+a(x+a)\geq4{\bigl(x'-{x'}^{2}\bigr)}^{2}\left(1-x'\right)+4{\bigl(x'-{x'}^{2}\bigr)}^{2}a\\
={\left(2\bigl(x'-{x'}^{2}\bigr)\right)}^{2}(1-x)\geq{\left(2\bigl(x'-{x'}^{2}\bigr)\right)}^{2}(1-x)^{2}.
\end{gathered}
\label{eq:aux:cbound}
\end{equation}
 Putting $c=4{\bigl(x'-{x'}^{2}\bigr)}^{2}\left(1-x'\right)+a(x+a)$
in \eqref{eq:g^-f^} and then using \eqref{eq:aux:sqrtc coeff} and
\eqref{eq:aux:cbound}, 
\begin{align*}
\tilde{g}(x,c,a,a)-\tilde{f}(x,c) & =-x(1-x-2a)+(1-x-a)(x+a)^{2}\\
 & {}+\left(\sqrt{1-x'}-1\right)\sqrt{4{\bigl(x'-{x'}^{2}\bigr)}^{2}\left(1-x'\right)+a(x+a)}\displaybreak[0]\\
 & \leq-x(1-x-2a)+(1-x-a)(x+a)^{2}-x'\bigl(x'-{x'}^{2}\bigr)(1-x)\displaybreak[0]\\
 & =-x\left[(1-x-2a)-(1-x-a)(x+a)^{2}\right]\leq-x\left[(1-x-2a)-(1-x-a)\frac{1}{4}\right]\\
 & =-x\left(\frac{3-3x-7a}{4}\right),
\end{align*}
 which is $\leq0$ when $a\leq\frac{3-3x-3a}{4}=\frac{3-3x'}{4}$.
Assume $a>\frac{3-3x'}{4}$. Since $x'\geq a$, $x'>\frac{3}{7}$
(and we have also assumed $\frac{1}{2}\geq x'$); and 
\begin{align*}
c & =4{\bigl(x'-{x'}^{2}\bigr)}^{2}\left(1-x'\right)+a(x+a)>4{\bigl(x'-{x'}^{2}\bigr)}^{2}\left(1-x'\right)+\frac{(3-3x')x'}{4}\\
 & >4\cdot{\left(\frac{3}{7}-\left(\frac{3}{7}\right)^{2}\right)}^{2}\left(1-\frac{1}{2}\right)+\frac{\left(3-3\cdot\frac{1}{2}\right)\cdot\frac{3}{7}}{4}=\frac{5391}{19208}>\frac{1}{4}
\end{align*}
 contrary to our assumption that $c\leq\frac{c}{1-x-a}\leq\frac{1}{4}$.
\item \label{case:g:midc}When $c\leq\frac{1}{4}<\frac{c}{1-x-a}$ (and
recall $4{\bigl(x'-{x'}^{2}\bigr)}^{2}\left(1-x'\right)+a(x+a)\leq c$),
using \eqref{eq:fbound},
\begin{equation}
\begin{aligned}g(x,c,a,a)-f(x,c) & \leq\tilde{g}(x,c,a,a)-\tilde{f}(x,c)=a+(1-x-a)\left((x+a)^{2}-2(x+a)+1.25\right)\\
 & {}+2a(1-x-a)-\left(x^{2}-2x+1-c+\sqrt{c}\right)\\
 & =-x(1-x-2a)+(1-x-a)\left((x+a)^{2}+0.25\right)+c-\sqrt{c}.
\end{aligned}
\label{eq:auxcase112}
\end{equation}

\begin{enumerate}[label=\labeliv, ref=\refiv, labelsep=0em, leftmargin=1em, labelwidth=\labelwidthiv, itemindent=\labelwidth, align=left]
\item If $4{\bigl(x'-{x'}^{2}\bigr)}^{2}\left(1-x'\right)+a(x+a)\leq\frac{1}{4}(1-x-a)$
(which is $<c$), 

since $t-\sqrt{t}$ is decreasing in $0\le t\le\frac{1}{4}$, replacing
$c$ by $\frac{1}{4}(1-x-a)$ in \eqref{eq:auxcase112}, we get
\begin{gather*}
g(x,c,a,a)-f(x,c)\leq-x(1-x-2a)+(1-x-a)\left((x+a)^{2}+0.25\right)+\frac{1}{4}(1-x-a)\\
{}-\sqrt{\frac{1}{4}(1-x-a)}=\tilde{g}\!\left(x,\frac{1}{4}(1-x-a),a,a\right)-\tilde{f}\!\left(x,\frac{1}{4}(1-x-a)\right)\leq0,
\end{gather*}
 as it falls in Case \ref{case:g:smallc} above.
\item If $\frac{1}{4}(1-x-a)\leq4{\bigl(x'-{x'}^{2}\bigr)}^{2}\left(1-x'\right)+a(x+a)$
(which is $\leq c$),%

again, by \eqref{eq:auxcase112} we have, 
\begin{equation}
\begin{gathered}g(x,c,a,a)-f(x,c)\leq-x(1-x-2a)+(1-x-a)\left((x+a)^{2}+0.25\right)\\
{}+4{\bigl(x'-{x'}^{2}\bigr)}^{2}\left(1-x'\right)+a(x+a)-\sqrt{4{\bigl(x'-{x'}^{2}\bigr)}^{2}\left(1-x'\right)+a(x+a)}.
\end{gathered}
\label{eq:aux0}
\end{equation}
Let $b=\max\!\left(4{\bigl(x'-{x'}^{2}\bigr)}^{2}\left(1-x'\right)+{x'}^{2}-\frac{1}{4},0\right)$.
Now some auxiliary calculations follow. Since $\frac{\mathrm{d}}{\mathrm{d}t}\sqrt{t}=\frac{1}{2\sqrt{t}}$,
and by \eqref{eq:aux:cbound}, 
\begin{equation}
\begin{gathered}\sqrt{4{\bigl(x'-{x'}^{2}\bigr)}^{2}\left(1-x'\right)+{x'}^{2}-b}-\sqrt{4{\bigl(x'-{x'}^{2}\bigr)}^{2}\left(1-x'\right)+a(x+a)}\\
\leq\frac{1}{2\sqrt{4{\bigl(x'-{x'}^{2}\bigr)}^{2}\left(1-x'\right)+a(x+a)}}\bigl(x(x+a)-b\bigr)\leq\frac{1}{4\left(x'-{x'}^{2}\right)(1-x)}\bigl(x(x+a)-b\bigr).
\end{gathered}
\label{eq:aux1}
\end{equation}
\[
(1-a)\cdot4\left(x'-{x'}^{2}\right)(1-x)\geq4\left(x'-{x'}^{2}\right)(1-x')=4{(1-x')}{}^{2}x'\geq4\cdot\left(\frac{1}{2}\right)^{2}x'=x+a.
\]
So,
\begin{equation}
\begin{gathered}-x(1-x-2a)+\left(\frac{1}{4\left(x'-{x'}^{2}\right)(1-x)}-1\right)\bigl(x(x+a)-b\bigr)\\
\leq-x(1-x-2a)+\left(\frac{1-a}{x+a}-1\right)x(x+a)=0.
\end{gathered}
\label{eq:aux2}
\end{equation}
By \eqref{eq:aux_x'}, $4{\bigl(x'-{x'}^{2}\bigr)}^{2}\leq x'$, so
\begin{equation}
4{\bigl(x'-{x'}^{2}\bigr)}^{2}\left(1-x'\right)+{x'}^{2}\geq4{\bigl(x'-{x'}^{2}\bigr)}^{2}\left(1-x'\right)+4{\bigl(x'-{x'}^{2}\bigr)}^{2}x'={\left(2\bigl(x'-{x'}^{2}\bigr)\right)}^{2}.\label{eq:aux:h cbound}
\end{equation}
\begin{gather}
4{\bigl(x'-{x'}^{2}\bigr)}^{2}\left(1-x'\right)+{x'}^{2}-b\geq{\left(2\bigl(x'-{x'}^{2}\bigr)\right)}^{2},\label{eq:aux3}
\end{gather}
since if $b>0$, $4{\bigl(x'-{x'}^{2}\bigr)}^{2}\left(1-x'\right)+{x'}^{2}-b=\frac{1}{4}\geq{\left(2\bigl(\frac{1}{4}-{\left(\frac{1}{2}-x'\right)}^{2}\bigr)\right)}^{2}={\left(2\bigl(x'-{x'}^{2}\bigr)\right)}^{2}$
(if b = 0, then it holds by \eqref{eq:aux:h cbound}). Now using \eqref{eq:aux1}
in \eqref{eq:aux0}, and then using \eqref{eq:aux2} and \eqref{eq:aux3}
and that $t-\sqrt{t}$ is decreasing in $0\le t\le1/4$, we get
\begin{gather*}
g(x,c,a,a)-f(x,c)\leq(1-x-a)\left((x+a)^{2}+0.25\right)+4{\bigl(x'-{x'}^{2}\bigr)}^{2}\left(1-x'\right)+{x'}^{2}-b\\
{}-\sqrt{4{\bigl(x'-{x'}^{2}\bigr)}^{2}\left(1-x'\right)+{x'}^{2}-b}-x(1-x-2a)+\left(\frac{1}{4\left(x'-{x'}^{2}\right)(1-x)}-1\right)\bigl(x(x+a)-b\bigr)\\
\leq(1-x-a)\left((x+a)^{2}+0.25\right)+{\left(2\bigl(x'-{x'}^{2}\bigr)\right)}^{2}-2\bigl(x'-{x'}^{2}\bigr)\\
=4\left(x'-\frac{1}{4}\right)\left(x'-\frac{1}{2}\right)^{2}\left(x'-1\right),
\end{gather*}
 which is $\leq0$ when $\frac{1}{4}\leq x'\leq1$. When $x'<\frac{1}{4}$,
we show that this subcase cannot hold: 
\begin{equation}
\begin{gathered}4{\bigl(x'-{x'}^{2}\bigr)}^{2}\left(1-x'\right)+a(x+a)-\frac{1}{4}(1-x-a)<\left(4\cdot\left(\frac{1}{4}-\left(\frac{1}{4}\right)^{2}\right)^{2}-\frac{1}{4}\right)(1-x')+{x'}^{2}\\
=-\frac{7}{64}(1-x')+{x'}^{2}<-\frac{1}{12}(1-x')+{x'}^{2}=\left(x'-\frac{1}{4}\right)\left(x'+\frac{1}{3}\right)<0.
\end{gathered}
\label{eq:indirect}
\end{equation}
\end{enumerate}
\item When $\frac{1}{4}\leq c$ (which is $<\frac{c}{1-x-a}$), 
\begin{align*}
g(x,c,a,a)-f(x,c) & \leq\tilde{g}(x,c,a,a)-f(x,c)=a+(1-x-a)\left((x+a)^{2}-2(x+a)+1.25\right)\\
 & {}+2a(1-x-a)-\left(x^{2}-2x+1.25\right)=\tilde{g}\!\left(x,\frac{1}{4},a,a\right)-\tilde{f}\!\left(x,\frac{1}{4}\right)\leq0,
\end{align*}
 as it falls in Case \ref{case:g:smallc} or \ref{case:g:midc} above.
\end{enumerate}
\item When $x'\leq\frac{1}{2}$ and $c'\leq4{\bigl(x'-{x'}^{2}\bigr)}^{2}$,
\[
g(x,c,a,a)=a+(1-x-a)\left(1-x'+\left(\frac{1}{4\left(x'-{x'}^{2}\right)}-1\right)c'\right)+2a(1-x-a),
\]
 which is linear in $c'$, so also in $c$. $f(x,c)$ is concave in
$c$, so it is enough to check that $g$ is smaller than $f$ in the
ends of the interval $c\in\left[a(a+x),4{\bigl(x'-{x'}^{2}\bigr)}^{2}\left(1-x'\right)+a(x+a)\right]$:
\begin{align*}
g\!\left(x,a(a+x),a,a\right) & =a+(1-x-a)^{2}+2a(1-x-a)\\
 & \leq a+(1-x-a)^{2}+2a(1-x-a)+\frac{2ax\left(\frac{1}{2}-x-a+\left(1-\frac{a}{2}\right)x\right)}{\sqrt{a(x+a)}+ax+a}\\
 & =x^{2}-2x+1-a(x+a)+\sqrt{a(x+a)}\leq f\bigl(x,a(x+a)\bigr),
\end{align*}
 since $\frac{1}{2}-x-a\geq0$ and $a(x+a)\leq{x'}^{2}\leq\frac{1}{4}$,
and using \eqref{eq:fbound}. Whereas the higher end of the interval
was handled above in Case \ref{case:g:smallcprime} since $f$ is
continuous. 
\item Finally, when $\frac{1}{2}\leq x'$, 
\[
g(x,c,a,a)=a+(1-x-a)^{2}+2a(1-x-a)=x^{2}-2x+1-a^{2}+a.
\]
If $x\leq\frac{1}{2}$, let $\tilde{c}=\min\!\left(c,\frac{1}{4}\right)$.
Since $c\geq a(x+a)\geq a^{2}$, and by \eqref{eq:fbound} 
\[
g(x,c,a,a)=x^{2}-2x+1-a^{2}+a\leq x^{2}-2x+1-\tilde{c}+\sqrt{\tilde{c}}\leq f(x,c).
\]
If $\frac{1}{2}\leq x$, then $a\leq1-x\leq\frac{1}{2}$. $-a^{2}+a$
is monotonously increasing in $a\in\left[0,\frac{1}{2}\right]$, so
\[
g(x,c,a,a)=x^{2}-2x+1-a^{2}+a\leq x^{2}-2x+1-(1-x)^{2}+1-x=1-x\leq f(x,c),
\]
by Point \ref{enu:1-x}.
\end{enumerate}
\item Now consider $a\geq\frac{c}{x+a}$, that is, $a\geq\frac{-x+\sqrt{x^{2}+4c}}{2}$.
Then $\tilde{a}=\frac{c}{x+a}\leq\frac{c}{x+\frac{-x+\sqrt{x^{2}+4c}}{2}}=\frac{-x+\sqrt{x^{2}+4c}}{2}$,
and $\frac{c-\tilde{a}(x+a)}{1-x-a}=0$, so $f\!\left(x+a,\frac{c-\tilde{a}(x+a)}{1-x-a}\right)=1-x-a$.
\[
g(x,c,a,\tilde{a})=a+(1-x-a)^{2}+2\frac{c}{x+a}(1-x-a)\leq a+(1-x-a)^{2}+\left(-x+\sqrt{x^{2}+4c}\right)(1-x-a),
\]
 which is quadratic in $a$ with a positive leading coefficient, so
its maximum is at one end of the interval $\left[\frac{-x+\sqrt{x^{2}+4c}}{2},1-x\right]$.
$a=\frac{-x+\sqrt{x^{2}+4c}}{2}$ (i.e.,\ $a=\frac{c}{x+a}=\tilde{a}$)
was handled above in Case \ref{case:g:a=00003Datilde} If $a=1-x$,
the right side of the inequality equals $1-x$ which is $\leq f(x,c)$
by Point \ref{enu:1-x}.\qedhere
\end{enumerate}
\end{boldproof}

\begin{boldproof}[Proof of Point \ref{enu:h}.]
If $x+a=0$, $h(x,c,a,\tilde{a})=f(x,c)$. From now on, we assume
that $x+a>0.$

We first show that $h$ is monotonously increasing in $\tilde{a}$.
Using \eqref{eq:particalcforpoint4} and a calculation similar to
the one in the proof of Point \ref{enu:g}, we have 
\begin{gather*}
\frac{\partial}{\partial\tilde{a}}h\!\left(x,c,a,\tilde{a}\right)=2(1-x-a)+(1-x-a)\cdot\left(\frac{\partial}{\partial c}f\right)\!\left(x+a,\frac{c-(x+\tilde{a})(x+a)}{1-x-a}\right)\\
{}\cdot\frac{\partial}{\partial\tilde{a}}\!\left(\frac{c-(x+\tilde{a})(x+a)}{1-x-a}\right)\geq2(1-x-a)-(x+a)\left(\begin{cases}
\frac{1}{4\left((x+a)-(x+a)^{2}\right)}-1 & \tif x+a\leq\frac{1}{2}\\
0 & \tif\frac{1}{2}\leq x+a
\end{cases}\right)\geq0.
\end{gather*}

Therefore, from now on we assume $\tilde{a}=\min\!\left(a,\frac{c}{x+a}-x\right)$.

\begin{enumerate}[label=\labeli, ref=\refi, labelsep=0em, leftmargin=0em, labelwidth=\labelwidthi, itemindent=\labelwidth, align=left]
\item \label{case:h:a=00003Datilde}First assume $a\leq\frac{c}{x+a}-x$
(so \foreignlanguage{english}{$\tilde{a}=a$}), which is equivalent
to $(x+a)^{2}\leq c$ or $a\leq\sqrt{c}-x$. Let $x'=x+a$ and $c'=\frac{c-(x+a)^{2}}{1-x-a}$.

\begin{enumerate}[label=\labelii, ref=\refii, labelsep=0em, leftmargin=1em, labelwidth=\labelwidthii, itemindent=\labelwidth, align=left]
\item \label{case:h:smallcprime}When $x'\leq\frac{1}{2}$ and $4{\bigl(x'-{x'}^{2}\bigr)}^{2}\leq c'$,
we bound $h$ from above: 
\begin{gather*}
h(x,c,a,a)=a+(1-x-a)f\!\left(x+a,\frac{c-(x+a)^{2}}{1-x-a}\right)+2a(1-x-a)+x-3x(x+a)\\
\leq a+(1-x-a)f\!\left(x+a,\frac{c}{1-x-a}\right)+2a(1-x-a)+x-3x(x+a)=:\tilde{h}(x,c,a,a).
\end{gather*}
We now consider subcases based on the values of $c$ and $\frac{c}{1-x-a}$
compared to $\frac{1}{4}$. Note that $4{\bigl(x'-{x'}^{2}\bigr)}^{2}\leq c'\leq\frac{c}{1-x-a}$.

\begin{enumerate}[label=\labeliii, ref=\refiii, labelsep=0em, leftmargin=1em, labelwidth=\labelwidthiii, itemindent=\labelwidth, align=left]
\item \label{case:h:smallc}When $c\le\frac{c}{1-x-a}\leq\frac{1}{4}$,
using \eqref{eq:fbound}, 
\begin{equation}
\begin{gathered}h(x,c,a,a)-f(x,c)\leq\tilde{h}(x,c,a,a)-\tilde{f}(x,c)=a+(1-x-a)\biggl((x+a)^{2}-2(x+a)+1\\
{}-\frac{c}{1-x-a}+\frac{\sqrt{c}}{\sqrt{1-x-a}}\biggr)+2a(1-x-a)+x-3x(x+a)-\left(x^{2}-2x+1-c+\sqrt{c}\right)\\
\begin{gathered}=-2x^{2}-xa+(1-x-a)(x+a)^{2}+\left(\sqrt{1-x'}-1\right)\sqrt{c}.\end{gathered}
\end{gathered}
\label{eq:hputtingc}
\end{equation}
\[
\frac{\partial}{\partial c}\!\left(\tilde{h}(x,c,a,a)-\tilde{f}(x,c)\right)=\frac{\sqrt{1-x-a}-1}{2\sqrt{c}}\leq0.
\]
 So it is enough to check that $\tilde{h}(x,c,a,a)-\tilde{f}(x,c)\leq0$
when $c'=4{\bigl(x'-{x'}^{2}\bigr)}^{2}$ or, equivalently, when $c=4{\bigl(x'-{x'}^{2}\bigr)}^{2}\left(1-x'\right)+{x'}^{2}$;
then it is also $\leq0$ for bigger $c$. As seen in \eqref{eq:aux:h cbound}
and \eqref{eq:aux:sqrtc coeff} in the proof of Point \ref{enu:g},
$4{\bigl(x'-{x'}^{2}\bigr)}^{2}\left(1-x'\right)+{x'}^{2}\geq{\left(2\bigl(x'-{x'}^{2}\bigr)\right)}^{2}$,
and $\sqrt{1-x'}-1\leq-\frac{x'}{2}\leq0$. Putting $c=4{\bigl(x'-{x'}^{2}\bigr)}^{2}\left(1-x'\right)+{x'}^{2}$
in \eqref{eq:hputtingc} and using these inequalities, we get 
\begin{align*}
\tilde{h}(x,c,a,a)-\tilde{f}(x,c) & =-2x^{2}-xa+(1-x-a)(x+a)^{2}\\
 & {}+\left(\sqrt{1-x'}-1\right)\sqrt{4{\bigl(x'-{x'}^{2}\bigr)}^{2}\left(1-x'\right)+{x'}^{2}}\\
 & \leq-2x^{2}-xa+(1-x-a)(x+a)^{2}-x'\bigl(x'-{x'}^{2}\bigr)=-2x^{2}-xa\leq0.
\end{align*}
\item \label{case:h:midc}When $c\leq\frac{1}{4}<\frac{c}{1-x-a}$ (and
recall ${\left(2\bigl(x'-{x'}^{2}\bigr)\right)}^{2}\leq4{\bigl(x'-{x'}^{2}\bigr)}^{2}\left(1-x'\right)+{x'}^{2}\leq c$),
using \eqref{eq:fbound}, 
\begin{equation}
\begin{aligned}h(x,c,a,a)-f(x,c) & \leq\tilde{h}(x,c,a,a)-\tilde{f}(x,c)=a+(1-x-a)\left((x+a)^{2}-2(x+a)+1.25\right)\\
 & {}+2a(1-x-a)+x-3x(x+a)-\left(x^{2}-2x+1-c+\sqrt{c}\right)\\
 & \begin{gathered}=-2x^{2}-xa+(1-x-a)\left((x+a)^{2}+0.25\right)+c-\sqrt{c}.\end{gathered}
\end{aligned}
\label{eq:hdecreasec}
\end{equation}

\begin{enumerate}[label=\labeliv, ref=\refiv, labelsep=0em, leftmargin=1em, labelwidth=\labelwidthiv, itemindent=\labelwidth, align=left]
\item If $4{\bigl(x'-{x'}^{2}\bigr)}^{2}\left(1-x'\right)+{x'}^{2}\leq\frac{1}{4}(1-x-a)$
(which is $<c$), 

since $t-\sqrt{t}$ is decreasing in $0\le t\le\frac{1}{4}$, replacing
$c$ by $\frac{1}{4}(1-x-a)$ in \eqref{eq:hdecreasec}, we get 
\begin{align*}
h(x,c,a,a)-f(x,c) & \leq-2x^{2}-xa+(1-x-a)\left((x+a)^{2}+0.25\right)+\frac{1}{4}(1-x-a)\\
 & {}-\sqrt{\frac{1}{4}(1-x-a)}=\tilde{h}\!\left(x,\frac{1}{4}(1-x-a),a,a\right)-\tilde{f}\!\left(x,\frac{1}{4}(1-x-a)\right)\leq0,
\end{align*}
 as it falls in Case \ref{case:h:smallc} above.
\item If $\frac{1}{4}(1-x-a)\leq4{\bigl(x'-{x'}^{2}\bigr)}^{2}\left(1-x'\right)+{x'}^{2}$
(which is $\leq c$),

again, by \eqref{eq:hdecreasec} (and since $-2x^{2}-xa\leq0$), we
get
\begin{align*}
h(x,c,a,a)-f(x,c) & \leq(1-x-a)\left((x+a)^{2}+0.25\right)+{\left(2\bigl(x'-{x'}^{2}\bigr)\right)}^{2}-2\bigl(x'-{x'}^{2}\bigr)\\
 & =4\left(x'-\frac{1}{4}\right)\left(x'-\frac{1}{2}\right)^{2}\left(x'-1\right),
\end{align*}
 which is $\leq0$ when $\frac{1}{4}\leq x'\leq1$. When $x'<\frac{1}{4}$,
we show that this subcase cannot hold: 
\[
4{\bigl(x'-{x'}^{2}\bigr)}^{2}\left(1-x'\right)+{x'}^{2}-\frac{1}{4}(1-x-a)<\left(4\cdot\left(\frac{1}{4}-\left(\frac{1}{4}\right)^{2}\right)^{2}-\frac{1}{4}\right)(1-x')+{x'}^{2}<0,
\]
like in \eqref{eq:indirect} in the proof of Point \ref{enu:g}.
\end{enumerate}
\item When $\frac{1}{4}\leq c$ (which is $<\frac{c}{1-x-a}$), 
\begin{gather*}
h(x,c,a,a)-f(x,c)\leq\tilde{h}(x,c,a,a)-f(x,c)=a+(1-x-a)\left((x+a)^{2}-2(x+a)+1.25\right)\\
{}+2a(1-x-a)+x-3x(x+a)-\left(x^{2}-2x+1.25\right)=\tilde{h}\!\left(x,\frac{1}{4},a,a\right)-\tilde{f}\!\left(x,\frac{1}{4}\right)\leq0,
\end{gather*}
 as it falls in Case \ref{case:h:smallc} or \ref{case:h:midc} above.
\end{enumerate}
\item When $x'\leq\frac{1}{2}$ and $c'\leq4{\bigl(x'-{x'}^{2}\bigr)}^{2}$,
\[
h(x,c,a,a)=a+(1-x-a)\left(1-x'+\left(\frac{1}{4\left(x'-{x'}^{2}\right)}-1\right)c'\right)+2a(1-x-a)+x-3x(x+a),
\]
 which is linear in $c'$, so also in $c$. $f(x,c)$ is concave in
$c$, so it is enough to check that $h$ is smaller than $f$ in the
ends of the interval $c\in\left[(x+a)^{2},4{\bigl(x'-{x'}^{2}\bigr)}^{2}\left(1-x'\right)+{x'}^{2}\right]$:
\begin{align*}
h\!\left(x,(x+a)^{2},a,a\right) & =a+(1-x-a)^{2}+2a(1-x-a)+x-3x(x+a)\\
 & \leq a+(1-x-a)^{2}+2a(1-x-a)+x-3x(x+a)+x(2x+a)\\
 & =x^{2}-2x+1-(x+a)^{2}+(x+a)\leq f\!\left(x,(x+a)^{2}\right)
\end{align*}
 since $(x+a)^{2}={x'}^{2}\leq\frac{1}{4}$, and using \eqref{eq:fbound}.
Whereas the higher end of the interval was handled above in Case \ref{case:h:smallcprime}
since $f$ is continuous. 
\item Finally, when $\frac{1}{2}\leq x'$, $c\geq(x+a)^{2}\geq\frac{1}{4}$,
so 
\[
h(x,c,a,a)=a+(1-x-a)^{2}+2a(1-x-a)+x-3x(x+a)\leq x^{2}-2x+1-(x+a)^{2}+(x+a).
\]
 If $x\leq\frac{1}{2}$, then 
\[
h(x,c,a,a)\leq x^{2}-2x+1-(x+a)^{2}+(x+a)\leq x^{2}-2x+1.25=f(x,c).
\]
 If $\frac{1}{2}\leq x$, then $-(x+a)^{2}+(x+a)$ is monotonously
decreasing in $a$, so 
\[
h(x,c,a,a)\leq x^{2}-2x+1-(x+a)^{2}+(x+a)\leq x^{2}-2x+1-x^{2}+x=1-x=f(x,c).
\]
\end{enumerate}
\item Now we consider $a\geq\frac{c}{x+a}-x$, that is, $a\geq\sqrt{c}-x$.
Then $\tilde{a}=\min\!\left(a,\frac{c}{x+a}-x\right)=\frac{c}{x+a}-x\leq\sqrt{c}-x$,
and $\frac{c-(x+\tilde{a})(x+a)}{1-x-a}=0$, so $f\!\left(x+a,\frac{c-(x+\tilde{a})(x+a)}{1-x-a}\right)=1-x-a$.
\begin{align*}
h(x,c,a,\tilde{a}) & =a+(1-x-a)^{2}+2\left(\frac{c}{x+a}-x\right)(1-x-a)+x-3x(x+a)\\
 & \leq a+(1-x-a)^{2}+2\left(\sqrt{c}-x\right)(1-x-a)+x-3x(x+a),
\end{align*}
 which is quadratic in $a$ with a positive leading coefficient, so
its maximum is at one end of the interval $\left[\sqrt{c}-x,1-x\right]$.
$a=\sqrt{c}-x$ (i.e.,\ $a=\frac{c}{x+a}-x=\tilde{a}$) was handled
above in Case \ref{case:h:a=00003Datilde} If $a=1-x$, the right
side of the inequality equals $1-3x$ which is $\leq f(x,c)$ by Point
\ref{enu:1-x}.\qedhere
\end{enumerate}
\end{boldproof}

\end{document}